\documentclass[12pt,reqno]{amsart}

\usepackage{amssymb,hyperref}
\usepackage{amsmath}
\usepackage{etoolbox}
\usepackage{mathrsfs}

\newtheorem{theorem}{Theorem}[section]
\newtheorem{proposition}[theorem]{Proposition}
\newtheorem{lemma}[theorem]{Lemma}
\newtheorem{corollary}[theorem]{Corollary}
\newtheorem{question}[theorem]{Question}

\theoremstyle{definition}
\newtheorem{definition}[theorem]{Definition}
\newtheorem{remark}[theorem]{Remark}

\usepackage[usenames,dvipsnames,svgnames,table]{xcolor}
\usepackage{tikz}
\usepackage{tikz-cd}

\newcommand{\lra}{\longrightarrow}
\renewcommand{\Im}{\operatorname{Im}}

\newcommand{\ra}{\rightarrow}
\newcommand{\C}{\mathbb{C}}

\newcommand{\bb}{{\mathbb B}}
\newcommand{\tbb}{{\widetilde{\mathbb B}}}

\newcommand{\Z}{{\mathbb{Z}}}
\newcommand{\Aut}{\operatorname {Aut}}

\newcommand{\debar}{{\overline{\partial} } }

\newcommand{\Sym}{\operatorname{Sym}}

\newcommand{\OO}{\mathcal O}
\renewcommand{\SS}{{\widetilde {\mathcal S}}}
\newcommand{\Deltas}{\widetilde{\Delta}}

\numberwithin{equation}{section}

\begin{document}

\baselineskip=15pt

\title[Prym varieties and projective structures]{Prym 
varieties and projective structures on Riemann surfaces}

\author[I. Biswas]{Indranil Biswas}

\address{Department of Mathematics, Shiv Nadar University, NH91, Tehsil
Dadri, Greater Noida, Uttar Pradesh 201314, India}

\email{indranil.biswas@snu.edu.in, indranil29@gmail.com}

\author[A. Ghigi]{Alessandro Ghigi}
	
\address{Dipartimento di Matematica, Universit\`a di Pavia, via
Ferrata 5, I-27100 Pavia, Italy}
	
\email{alessandro.ghigi@unipv.it}
	
\author[L. Vai]{Luca Vai}
	
\address{Dipartimento di Matematica, Universit\`a di Pavia, via
Ferrata 5, I-27100 Pavia, Italy}
	
\email{luca.vai628@gmail.com, luca.vai01@universitadipavia.it}

\subjclass[2010]{14H10, 14H40, 14K25, 53B10}

\keywords{Prym variety, theta function, projective structure, moduli space of curves}

\begin{abstract} 
  Given an \'etale double covering
  $\pi\, :\, \widetilde{C}\, \longrightarrow\, C$ of compact Riemann
  surfaces with $C$ of genus at least two, we use the Prym variety of the cover to construct canonical projective structures on both
  $\widetilde C$ and $C$.  
  This construction can be interpreted as a section of an affine bundle over the moduli space of \'etale double
  covers.  
  The $\debar$--derivative of this section is a (1,1)--form
  on the moduli space. We compute this derivative in terms of Thetanullwert
  maps. 
  Using the Schottky--Jung identities we show that, in general,  the projective structure on $C$ depends on the cover.
\end{abstract}

\maketitle

\tableofcontents

\section{Introduction}

Let $C$ be a compact Riemann surface. A \emph{projective structure} on $C$ is the datum of a coordinate covering of $C$ such that all the transition functions are M\"obius transformations.
Two projective structures are considered the same if they are compatible, meaning their union is still a projective structure. 
Historically, the most important projective structure is the one given by the Uniformization Theorem. 
This projective structure is canonical, in the sense that its construction does not require making choices.

There are at least two other examples of canonical projective structure: one constructed via Hodge theory in \cite{BCFP}, another constructed using theta functions in \cite{BGV}. 

Although the complex analytic definition of a projective structure is very natural, studying them, for example comparing two different projective structures, is quite complicated.
The algebraic viewpoint --- presented for example in \cite{Tyu,GU,Hu} --- is useful for this.
The space $\mathcal{P}(C)$ of projective structures on $C$ has a natural structure of an affine space modelled on the complex vector space $H^0(C,\, 2K_C)$. One way to describe 
projective structures is using certain elements of 
$H^0(C\times C,\, K_{C\times C}(2\Delta))$ --- which are called ``symmetric bidifferentials of the second kind" in the classical terminology --- where $\Delta$ is the diagonal divisor in $C\times C$.

In the first part of the article, given a compact Riemann surface $C$ of genus $g\,\geq\, 2$, and an \'etale
double cover 
\begin{equation}
\label{pcc}
\pi\,:\,\widetilde{C} \,\longrightarrow \, C,
\end{equation} 
we use the theory of Prym varieties to
construct projective structures both on $\widetilde C$ and on $C$. 
Let us sketch briefly this construction, which is at the basis of the paper.

Given a principally polarized abelian variety $(A,\,\Theta)$, the line bundle $\mathcal{O}_A(2\Theta)$ is canonical, i.e., intrinsically defined, the 
complete linear system $\big\vert 2\Theta\big\vert$ is base-point free and the vector space $H^0(A,\, 2\Theta)$ admits a Hermitian inner product that is intrinsically defined up to multiplication by a positive constant. 
Denote by $\mathscr L$ the orthogonal complement 
in $H^0(A,\,2\Theta)$ 
to the space of sections that vanish at the origin. 
This is 
an intrinsically defined
1-dimensional subspace. 
In the paper \cite{BGV} this construction was applied to the Jacobian $A=JC$ of a curve $C$. It was proved that pulling back the sections in $\mathscr L$ via the difference map $(p,\,q) \,\longmapsto\, \OO_C(p-q)$ gives rise to elements in $H^0(C\times C, K_{C\times C} (2\Delta))$. Normalizing them appropriately one gets a projective structure on $C$.

This paper develops a Prym version of this construction.
Consider the Prym variety $P\, =\, P(\pi:\widetilde C\to C)$ of the double cover \eqref{pcc}. It is endowed with a natural principal polarization $\Xi$.
Therefore as above 
we get an intrinsically-defined line $\mathscr L \,\subset\, H^0(P,\,2\Xi)$.
Consider the Prym difference map:
$$\phi\ :\ \widetilde C\times \widetilde C\ \longrightarrow\ P,\ \ \, (p,\, q)\ 
\longmapsto\ \mathcal{O}_{\widetilde C}(p-\sigma(p)-q+\sigma(q)).$$
We prove that 
\begin{equation}
  \label{phistar}
\phi^\ast (2\Xi) \ \cong \ K_{\widetilde C\times \widetilde C}(2\widetilde \Delta-2\Sigma).
\end{equation}
where $\widetilde \Delta$ is the reduced diagonal in $\widetilde{C}\times \widetilde{C}$ and $\Sigma$ is the graph of the nontrivial involution $\sigma\, \in\, {\rm Gal}(\pi)$; see 
Theorem \ref{prop:fundamentalpullback}.
Consequently,
for $s\in \mathscr L$ the section
$\phi^\ast s$, after an appropriate normalization, gives rise to a projective structure $\beta^Q_\pi$ on $\widetilde C$.
Finally, we show that this projective structure on $\widetilde C$ is invariant with respect to $\sigma$, hence is the pull-back via $\pi$ of a unique projective structure on $C$
that we denote by $\beta^P_\pi$ (see Corollary \ref{cor:descent}). 

The last fact is somewhat surprising, since the Prym variety $P$ represents in some sense the part of the Jacobian $J(\widetilde C)$ which is anti-invariant with respect to the involution $\sigma$.

At this point two questions seem rather natural. 
 If we start from the curve $C$, the additional datum needed to construct the covering \eqref{pcc} is a nontrivial 2-torsion point $\eta \in JC[2]-\{0\}$. So the projective structure can be seen as a function of $(C,\eta)$.
 If instead we start from the curve $\widetilde{C}$, the additional datum needed to construct the covering is the fixed-point free involution $\sigma\in \Aut (\widetilde{C})$. So the projective structure $\beta^Q$ can be considered also as a function of $(\widetilde{C}, \sigma)$.
\begin{question}
\label{question:one}
Is the Prym projective structure $\beta^P$ on $C$ independent of the covering, in other words, is it
 independent of $\eta$?
\end{question}
\begin{question}
\label{question:two}
Is the Prym projective structure $\beta^Q$ on $\widetilde{C}$ independent of the covering, in other words, is it independent of the involution $\sigma$?
\end{question}

If Question \ref{question:one} has a positive answer, one might wonder whether it coincides with one of the previously known canonical projective structures, the uniformization, Hodge and Theta structures mentioned above. If the answer is negative, then $\beta^P$ is a different kind of object, hence something new.

Both questions seem rather hard to deal with explicitly, since the projective structures are quite elusive and it is not easy to determine the properties of a specific projective structure. In the rest of the article we take a different route, similar to that already taken in \cite{BCFP,BFPT,BGT,BGV,BV}:
instead of studying the projective structures $\beta^Q$ and $\beta^P$  obtained from a fixed covering, 
we  study them through their variation in moduli as $\widetilde C,\,C$ and $\pi$ vary. 
More precisely we compute the $\debar$--derivative of $\beta^P$ in an appropriate sense. 

Let us explain in more detail the contents of the rest of the paper.
As it is known \cite{ZT,BFPT,BGT}, a family of projective structures on $M_g$ can be seen as a section of a certain (orbifold) bundle. 
Indeed, let $V_g$ denote the moduli space parametrizing classes of pairs of the form $(C,\, \mathcal Z)$, where $C$ a compact Riemann surface of genus $g$ and $\mathcal Z$ is a projective structure on it.
The forgetful map $V_g\, \longrightarrow\, M_g$, that sends $[C,\, \mathcal Z]$ to $[C]$, is a holomorphic torsor on $M_g$ for the holomorphic cotangent bundle $\Omega^1_{M_g}$.
Canonical projective structures, which depend smoothly on the moduli point $[C]$, correspond to $\mathcal{C}^{\infty}$ sections of $V_g\, \longrightarrow\, M_g$.
Therefore, the holomorphic differential of a smooth section $\beta$ of $V_g$ gives a $C^\infty$ $(1,\, 1)$--form on $M_g$ (see \cite[Section 2]{BCFP}, \cite{BGT} or \S \, \ref{ss:projmarked} below for more details).
Computing the $\debar$--derivative is a very powerful technique for studying canonical projective structures.
For example, it can be used to prove that two projective structures are different; this idea has already been used, for example, in \cite{BCFP,BV}.

In the case of Jacobian, which was treated in the papers \cite{BGV} and \cite{BV}, the computation of the $\debar$--derivative was achieved by studying the pullback map
\begin{equation*}
  H^0(JC,2\Theta) \lra H^0(K_{C\times C} (2\Delta)).
\end{equation*}
The pullback of a section of $2\Theta$ can in fact be expressed in terms of its value and its second derivative at the origin, see Proposition 5.3 of \cite{BGV}. In the Prym case the pullback map 
\begin{equation*}
  H^0(P,2\Xi) \lra H^0(K_{\widetilde C\times \widetilde C} (2\widetilde \Delta-2\Sigma))
\end{equation*}
is not so well understood.
Although this map has been studied before (see \cite{IZ,IP}), no explicit description of it exists.
In fact, the proof of \cite[Prop.5.3]{BGV} relies heavily on Fay's trisecant formula (see \cite{Po}) and on its symmetries.
Geometrically, this formula says that the image of the map $JC\lra \mathbb{P}^N$ given by the complete linear system $|2\Theta|$ --- we think of it as the Kummer variety of $JC$ --- admits trisecant lines.
There is an analogous result for Prym varieties, in fact the Kummer of a Prym variety admits quadrisecant planes, see \cite{BD}.
However, as far as the authors know, a "quadrisecant identity" for Prym varieties that enjoys symmetry properties similar to those of the Fay trisecant identity does not exist.
This seems to be an obstacle in the full description of the Prym pullback map.
We hope we can return to the relation between the quadrisecants planes of the Kummer of Prym varieties and the pullback map $H^0(P,2\Xi) \lra H^0(K_{\widetilde C\times \widetilde C} (2\widetilde \Delta-2\Sigma))$ in the future.

In any case, it turns out that in order to compute the $\debar$--derivative of the Prym projective structure it is not necessary to fully describe the pullback map.
It is enough to have a description of the composition
\begin{equation*}
    H^0(P,2\Xi) \,\lra\, H^0(K_{\widetilde{C}\times \widetilde{ C}} (2\widetilde{\Delta})) \,
\lra\,
H^0(3\Deltas,K_{\widetilde{C}\times \widetilde{ C}} (2\widetilde{\Delta})\big\vert _{3\Deltas}).
\end{equation*}
This is what we are able to compute in the Prym case, see Lemma \ref{lemma:pull}.
The Lemma expresses the restriction of the pullback of a section of $2\Xi$ in terms of its value and its second derivatives at the origin.

We now describe the $\debar$--derivative of the Prym projective structure.
Denote by $R_g$ the moduli space   that parametrizes the \'etale double coverings $\pi\,:\, \widetilde{C}\, \longrightarrow\, C$ as above,  where $C$ is a compact Riemann surface of genus $g$ and $\widetilde C$ is a smooth connected curve (hence $\pi$ is a nontrivial covering). The forgetful map $\xi\,:\,R_g\,\longrightarrow\, M_g$ that sends  $[\pi:\widetilde{C}\ra C]$ to $[C]$ is an \'etale covering of $M_g$ of degree $2^{2g}-1$. The map 
$$\beta^P\ :\ R_g\ \longrightarrow\ V_g,\ \ \, [\pi: \widetilde C \ra C]\ \longmapsto\ [C, \beta^P_\pi ],$$
is a section of $\xi^\ast V_g$. Since $\xi^\ast V_g\, \longrightarrow\, R_g$ is an $\Omega_{R_g}^1$--torsor, one may again interpret $\debar \beta^P$ as a $(1,\, 1)$ form on $R_g$.
The main result  of the paper is the following
(see Theorem \ref{teo:debar} and Remark \ref{rem:debarwithpg}):
\begin{theorem}\label{thmi}
The $\debar$--derivative of $ \beta^P$ has the following expression:
 $$\debar \beta^P\ \, =\ \,8 \pi P_g^\ast \Theta_{g-1}^\ast \omega_{FS},$$
where $P_g\,:\, R_g\, \longrightarrow \, A_{g-1}$ is the Prym map, $\Theta_{g-1}\,:\,A_{g-1}(2,4)\, \longrightarrow\, \mathbb{P}^{2^{g-1}-1}$ is the second order Thetanullwert map
(see \eqref{Thg} for the definition), and $\omega_{FS}$ denotes the Fubini--Study metric on $\mathbb{P}^{2^{g-1}-1}$.
\end{theorem}

As an application of Theorem \ref{thmi}, we show that in general the answer to Question \ref{question:one} is negative.
Equivalently, the section $\beta^P$ of $\xi^\ast V_g$ is not the pullback of a section of $V_g$.
We prove this by showing that the form $\debar \beta^P$ does not descend to $M_g$, in other words, it is not the pullback of a form on $M_g$. 
In particular 
$\beta^P$ is a new object living on $R_g$ not on $M_g$ and thus it differs from the other previously known canonical projective structures.

Our proof  uses the
classical Schottky--Jung identities, as they are presented in \cite{FR}.
In fact, Theorem \ref{teo:debar} expresses $\debar \beta^P$ at a point $[\widetilde C\to C]$ in terms of the period matrix $\tau$ of the Prym variety $P(\widetilde C\to C)$.
To understand whether $\debar\beta^P$ descends to $M_g$ or not, it is necessary to express it in terms of the period matrix $\Pi$ of the  curve $C$.
The Schottky--Jung identities give a relation between $\tau $ and $\Pi$, that allows to do this computation. 
In this regard, the Torelli space $\widetilde M_g$ (see \S \, \ref {subs:family1})
plays a fundamental role: being a common cover of $R_g$ and $M_g$, it serves as a connection between these two moduli spaces. 
The final result
(see Corollary \ref{cor:equivalentdebar}) is the following expression on $\widetilde M_g$:
\begin{gather*}
\debar \beta^P\ =\ 8\pi \ \Pi^\ast (\Theta_g')^\ast\omega_{FS},
\end{gather*}
where $\Pi\,:\,\widetilde{M}_g\,\longrightarrow\, \mathbb{H}_g$ is the period map, $\Theta_g'\,: \,\mathbb{H}_g\,\longrightarrow\, \mathbb{P}^N$ is the Thetanullwert map defined in \eqref{b3} and $\omega_{FS}$ is the Fubini--Study metric on $\mathbb{P}^N$.
We give a quick sketch of the proof of the fact that $\debar \beta^P$ does not descend to $M_g$.
Fix 
a curve of genus $g-1$ with  simple Jacobian and let
 $\Pi'$ be a period matrix of $C'$. Consider 
$V_1\, =\, \mathbb{H}_1\times \{\Pi'\}$ and $V_2\,=\,\{\Pi'\}\times \mathbb{H}_1$.
These are submanifolds  of $\mathbb{H}_g$.
We show that the restriction of $(\Theta_g')^\ast \omega_{FS}$ to $V_1$ is zero, while the restriction to $V_2$ is non zero. 
The result follows because $V_1$ and $V_2$ are in the closure of the Torelli locus $\Pi(\widetilde M_g)\,\subset\, \mathbb{H}_g$. 
To make this into a proof, some technical points are needed, which are settled in Lemma \ref{lemma:limitsequence}.

Question \ref{question:two} remains open. 
This questions seems quite different from Question \ref{question:one}. It is not clear whether one can answer it by computing the $\debar$-derivative and using a degeneration argument. 
We leave this problem for future investigations.

Finally we would like to stress that  the construction of $\beta^P$ and $\beta^Q$ fits extremely well with the setting of moduli of curves, Jacobians and Prym varieties. 
It seems a very natural object relating to several classical and fundamental objects.

The paper is organized as follows.
In Section 2 we recall the parts of the theory of Prym varieties and of theory of projective structures that are necessary for this paper.
Next starting from the étale double cover  \eqref{pcc} we
construct the Prym projective structures $\beta^Q_\pi$ on $\widetilde C$ and $\beta^P_\pi$ on $C$.
In Section 3 we consider the Torelli space $\widetilde M_g$.
We study the differential of the Prym--period map $\widetilde M_g\longrightarrow \mathbb{H}_{g-1}$. This is a required step in the computation of  $\debar \beta^P$.
Then, we explain how one can define smooth families of projective structures on $\widetilde M_g$.
In Section 4 we compute $\debar \beta^P$ explicitly on $R_g$ and on its cover $\widetilde M_g$, in terms of the period matrix of the Prym variety $P(\widetilde C\longrightarrow C)$.
In Section 5 we recall the  Schottky-Jung identities and use them to deduce the expression of $\debar \beta^P$ in terms of the period matrix of the bottom  $C$.
This requires the machinery  of theta functions.
Finally we prove that $\debar \beta^P$ does not descend to $M_g$ and that $\beta^P$ is not the pullback of a canonical projective structure on $M_g$. 

\section*{Acknowledgements}

The authors are deeply grateful to Juan Carlos Naranjo, who generously shared with them his knowledge of Prym varieties.
Without his  help it would have been impossible to start this project. 
They would also like to thank Riccardo Salvati Manni for helping   with his insights on theta functions and Gian Pietro Pirola for interesting discussions. 
The first author is partially supported by a J. C. Bose Fellowship (JBR/2023/000003).
The second and third authors were partially supported by INdAM-GNSAGA, by MIUR PRIN 2022:
20228JRCYB, ``Moduli spaces and special varieties'' and by FAR 2016
(Pavia) ``Variet\`a algebriche, calcolo algebrico, grafi orientati e
topologici''.

\section{The Prym projective structures}

This Section starts by
 recalling the definition and several useful facts about projective structures on Riemann surfaces.
In \ref{subs:prym} we recall some definitions and basic facts on Prym varieties.  
Then we construct the Prym projective structure on $\widetilde C$.
In \ref{proonbase}
we show that this projective structure is induced by a projective structure on $C$.

\subsection{Preliminaries on projective structures}
\label{subs:projectivestructures}

Throughout $C$ will denote a compact Riemann
surface of genus $g$, with $g\,\geq \,2$. The canonical line bundle of $C$ will be denoted by $K_C$; the same notation $K_M$ will be used for the canonical bundle of any smooth projective variety $M$. Set
\begin{equation*}
 S\ :=\ C\times C.
\end{equation*}
For $i\,=\,1,\, 2$, let
\begin{equation}\label{epi0}
p_i\ :\ S\ \longrightarrow\ C
\end{equation}
the natural projection to the $i$--th factor, let 
\begin{gather*}
 \Delta \ :=\ \{(x,\, x)\,\,\big\vert\,\, x\, \in\, C\}
\, \subset\, S
\end{gather*}
be the reduced diagonal, and let
\begin{equation}\label{et}
\textbf{t} \ :\ S \ \longrightarrow\ S, \ \ \, (x,\, y)\ \longrightarrow\
(y,\, x)
\end{equation}
be the involution of $S$.

We recall the definition of a projective structure:

\begin{definition}
A holomorphic atlas $\{(U_i,\, z_i)\}_{i\in I}$ of $C$ is called {\it projective} if all its transition
functions are M\"obius transformations, i.e., they satisfy the condition
\begin{equation}
\label{eq:projatlas}
z_i\ =\ \frac{a_{ij}z_j+b_{ij}}{c_{ij}z_j+d_{ij}}
\end{equation}
for all $i,\, j\, \in\, I$ with $U_i\cap U_j\,\not=\, \emptyset$, where
$a_{ij}$,\, $b_{ij}$,\, $c_{ij}$ and $d_{ij}$ are locally constant functions on
$U_i\cap U_j$. Two projective atlases are considered equivalent if their union is also a projective atlas. A
{\it projective structure} on $C$ is an equivalence class of projective atlases. Since each equivalence class contains a unique maximal atlas, a projective structure can also be defined as a maximal projective atlas.
\end{definition}

The space $\mathcal{P}(C)$ of all projective structures on $C$ is an affine space modelled on 
the complex vector space $H^0(C,\, 2K_C)$. This affine structure is presented in \cite[Section 2]{Tyu}; see also \cite{GU}. We will recall this construction in the proof of Proposition \ref{prop:pullbackaffine}. Here, we show how to describe algebraically projective 
structures. 
The line bundle 
\begin{gather}
\label{defbb}
\bb\ :=\ K_S(2\Delta)\ =\ K_S\otimes_{{\mathcal O}_S}\mathcal{O}_S(2\Delta) \ \lra\ S 
\end{gather}
will play a crucial role.
The restriction
$\bb\big\vert_{\Delta}$ is canonically trivial, and there is a unique section of $\bb\big\vert_{2\Delta}$ that is anti-invariant under $\textbf{t}$ (see
\eqref{et}) and whose restriction to the diagonal is $1\,\in\, H^0(\bb\big\vert_{\Delta})\,=\,\mathbb{C}$. This canonical section will be denoted by
\begin{equation}
\label{a1}
s_0\ \in\ H^0(2\Delta,\, \bb\big\vert_{2\Delta}).
\end{equation}
We note that $s_0$ is uniquely characterized by the following two properties: $s_0\big\vert_\Delta\,=\, 1$ and $s_0$ extends to a holomorphic section of $\bb$
\cite[p.~756, Proposition 2.10]{BR1}.

Consider the short exact sequence of sheaves on $3\Delta$
$$
0\,\longrightarrow\, K_S\big\vert_\Delta
\,\longrightarrow\, \bb\big\vert_{3\Delta}\,
\longrightarrow\, \bb\big\vert_{2\Delta}\,\longrightarrow\, 0,
$$
and the associated long exact sequence of cohomologies
$$0\,\longrightarrow\, H^0(C,\, 2K_C)\,=\, H^0(C,\, K_S\big\vert_\Delta)
\,\longrightarrow\, H^0(3\Delta,\, \bb\big\vert_{3\Delta})
\xrightarrow{\Psi}\, H^0(2\Delta,\, \bb\big\vert_{2\Delta})\,\longrightarrow\, 0;$$
note that $H^1(C,\, 2K_C)\,=\,0$ because $\text{genus}(C)\, \geq\, 2$.

There is the following canonical isomorphism between the affine spaces $\mathcal{P}(C)$ and
\begin{equation}\label{evc}
\mathcal{V}_C\ \, :=\ \, \Psi^{-1}(s_0)
\end{equation}
\cite[Section 3]{BR1}: given a projective atlas $\mathcal{Z}\,=\,\{(z_i,\,U_i)\}_{i\in I}$ we may construct a holomorphic section
\begin{equation*}
s_{\mathcal{Z}}\,\in\, H^0\left(\bigcup_{i\in I} U_i\times U_i, K_{C\times C}(2\Delta)\right),
\ \ \, s_{\mathcal{Z}}\big\vert_{U_i\times U_i}\,:=\,\frac{d p_1^\ast z_i\wedge dp_2^\ast z_i}{(p_1^\ast z_i-p_2^\ast z_i)^2}
\end{equation*}
(see \eqref{epi0} for $p_1,\, p_2$),
which is well defined, i.e., its pieces "patch together" compatibly, because the transition maps $z_i\,\longrightarrow\, z_j$ are of the form \eqref{eq:projatlas}.
The map 
\begin{equation}
\label{eq:Fmap}
F\,:\, \mathcal{P}(C) \,\longrightarrow\,\mathcal{V}_C, \ \ \, \mathcal{Z}\,\longmapsto \, (s_{\mathcal{Z}})\big\vert_{3\Delta}
\end{equation}
is a bijection between the projective structures $\mathcal{P}(C)$ and $\mathcal{V}_C$ in \eqref{evc} \cite[Theorem 3.2]{BR1}.

\begin{lemma}
\label{Fequivanriant}
The holomorphic automorphism group $\Aut( C)$ of $C$ acts on the left of both $\mathcal P(C)$ and
$\mathcal V_C$. The map $F$ constructed in \eqref{eq:Fmap} is equivariant for these actions. 
\end{lemma}

\begin{proof} 
Any $f\, \in\, \Aut( C)$ acts on $\mathcal P(C)$ by sending any
$\mathcal Z\,=\,\{(U_i,\, z_i)\}_{i \in I} \,\in\, \mathcal P(C)$ to
\begin{gather*}
f\cdot \mathcal Z\ =\ \{ (f (U_i), \,z_i \circ f^{-1} )\}_{i \in I}.
\end{gather*}

On the other hand, the diagonal action of $\Aut (C)$ on $S\,=\, C\times C$ obviously preserves $\Delta$. From this we get an action of $\Aut (C)$ on $\bb$ and on $ \bb \big\vert_{k\Delta}$ for every $k\,\geq\, 0$: if $(U,\,z)$ is a coordinate function on $C$, then
for the coordinate function $(z,\,z)$ on $U\times U$ on $S$ we have
\begin{gather*}
 f \cdot \frac{d p_1^\ast z\wedge dp_2^\ast z}{(p_1^\ast z-p_2^\ast z)^k}\ =\
 \frac{d p_1^\ast (f^{-1})^* z\wedge dp_2^\ast( f^{-1})^* z}{(p_1^\ast (f^{-1})^* z-p_2^\ast( f^{-1})^* z)^k}.
\end{gather*}
The action of $f$ on $H^0(2\Delta,\, \bb\big \vert_{2\Delta})$ obviously fixes the section $s_0$ in \eqref{a1}, and hence we get an action of $f$ on $\mathcal V_C$. It is easy to check that the map $F$ in \eqref{eq:Fmap} is equivariant.
\end{proof}

Let $(C,\,\mathcal Z)$ and $(C',\, \mathcal Z')$ be Riemann surfaces with projective structure. A local diffeomorphism $ f\,:\, C \,\longrightarrow\, C' $ is said to be a projective map (with respect to $\mathcal Z$ and $\mathcal Z'$) if given any charts $(U,\,z) \,\in\, \mathcal Z$ and $(V,\, w) \,\in\, \mathcal Z'$ such that $f(U) \,\subset\, V$, the composition $w \circ f \circ z^{-1}$ is the restriction of a M\"obius transformation. Note that this condition implies that
$f$ is holomorphic. Define
\begin{gather*}
\Aut (C, \mathcal Z)\ := \ \{f \,\in\, \Aut (C)\,\,\big\vert\,\,
f\cdot {\mathcal Z}\,=\, {\mathcal Z}\}.
\end{gather*}
In other words, $\Aut(C,\mathcal Z)$ is the stabilizer of $\mathcal Z$ for the action of $\Aut (C) $ on $\mathcal P(C)$. For a subgroup $G \,\subset\, \Aut(C)$ set
\begin{equation}
\label{projCG}
\mathcal P(C)^G\ :=\ \{ \mathcal Z\,\in\, \mathcal P(C)\,\,\big\vert\,\, G \,\subset\, \Aut (C,\mathcal Z)\}.
\end{equation}

The following proposition shows that projective structures behave well with respect to \'etale coverings, just like complex or Riemannian structures.

\begin{proposition}\label{prop:pullbackaffine}
Let $\pi\,:\, \widetilde{C}\, \longrightarrow \,C $ be a finite \'etale
covering map. Then the following four statements hold.
\begin{enumerate}
\item 
 For any $\mathcal Z \,\in\, \mathcal P(C)$, there is a unique projective structure $\pi^* \mathcal Z \,\in\, \mathcal P(\widetilde C)$ such that $\pi\,:\, (\widetilde C,\, \pi^*\mathcal Z)\, \longrightarrow \,(C,\,\mathcal Z)$ is a projective map.

\item The projective structure $\pi^ *\mathcal Z$ is characterized by the following
property: if $U\,\subset\, C$ is evenly covered, $V$ is a connected component of $\pi^{-1}(U)$ and $z \,:\, U \,\longrightarrow\, \mathbb{C}$ is a coordinate function
in ${\mathcal Z}$, then $z \circ \pi\big\vert_{V} \,:\, V \,\longrightarrow\,\mathbb{C}$ is a coordinate function belonging to $\pi^*\mathcal Z$.
 
 \item The pullback is contravariant, in other words, if $f\,:\,X\,\longrightarrow\, Y$ and $h\,:\,Y\,
 \longrightarrow\,Z$ are finite \'etale covers, then 
$$(h\circ f)^\ast\,=\,f^\ast\circ h^\ast\,:\,\mathcal{P}(Z)\,\longrightarrow\, \mathcal{P}(X).$$

\item The map
 \begin{gather}
 \label{pullback-map}
\pi^*\ :\ \mathcal P(C)\ \lra\ \mathcal P(\widetilde C)
 \end{gather}
 is an injective affine morphism, and its derivative is the pull-back map $$\pi^*\,:\, H^0(C,\,2K_C)\,\longrightarrow\, H^0(\widetilde{C},\, 2K_{\widetilde C}).$$
\end{enumerate}
\end{proposition}

\begin{proof}
Clearly, the property in (1) is satisfied by at most one projective structure $\pi^*\mathcal Z$.
It is easy to verify that (2) actually defines a projective structure $\pi^*\mathcal Z$ on $\widetilde C$ that satisfies property (1).

(3) is obvious. It is also quite clear that the map in \eqref{pullback-map} is injective: since $\pi$ has to be a projective map, the projective structure on $\widetilde C$ determines that on $C$. 
 It remains to prove the second part of (4) which says that
\begin{equation}
\label{eq:affinepullbackthesis}
\pi^\ast(\mathcal{Z}+\gamma)\ =\ \pi^\ast\mathcal{Z}+\pi^\ast \gamma
\end{equation}
for every $\gamma \,\in\, H^0(C,\, 2K_C)$.

Recall the affine space structure of the space of projective structures on a given Riemann
surface (see \cite[p. 169ff]{GuLec} and the proof of Lemma 3.6 in \cite{BR1} for more details).
Given a projective atlas $\mathcal{Z}\,=\, \{(z_i,\,U_i)\}_{i\in I}$ on $C$, let $\gamma\, =\, \gamma_i(z_i) dz_i^2$ on $U_i$. Taking a refinement of the cover, if necessary, we can find solutions $h_i$ for the differential equations
$$\vartheta_{z_i} h_i(z_i)\ =\ \gamma_i(z_i),$$
where $\vartheta$ is the Schwarzian derivative.
Then $\mathcal{Z}+\gamma$ is the projective structure corresponding to the atlas
$$\mathcal{Z}+\gamma\ :=\ \{(h_i\circ z_i,\, U_i)\}_{i\in I}.$$
Assume that the $U_i$ are evenly covered by $\pi$ and $\pi^{-1}(U_i)\,=\,\sqcup_{k=1}^d V_{i,k}$. Then
\begin{equation}
\label{eq:affinepullback}
\pi^\ast \mathcal{Z}\,=\,(z_i\circ \pi,\, V_{i,k})_{i\in I, k=1,\cdots ,d},\ \ \, \pi^\ast (\mathcal{Z}+\gamma)\,=\,(h_i\circ z_i\circ \pi,\, V_{i,k})_{i\in I, k=1,\cdots ,d},
\end{equation}
and $\pi^\ast \gamma$ is locally $\gamma_i(\pi(z_i)) d (z_i\circ \pi)^2$ on $V_{i,k}$. 
Now, the equality $\vartheta_{z_i} h_i(z_i)\,=\,\gamma_i(z_i)$ implies that $\vartheta_{z_i\circ \pi} h_i(\pi(z_i))\,=\,\gamma_i(\pi(z_i))$. Consequently, for the affine space structure
of $\mathcal{P}(\widetilde C)$ modelled on $H^0(\widetilde{C},\, 2K_{\widetilde C})$, we have
$$\pi^\ast\mathcal{Z}+\gamma\,=\, (h_i\circ z_i\circ \pi,\, V_{i,k})_{i\in I,k=1,\cdots ,d}.$$
Comparison of this with \eqref{eq:affinepullback} proves \eqref{eq:affinepullbackthesis}.
\end{proof}

\begin{definition}
\label{definduced}
If $\widetilde{\mathcal Z} \,\in\,\mathcal{P}(\widetilde C)$ satisfies the
condition $\widetilde{\mathcal Z}\,=\,\pi^\ast\mathcal Z$, we will say that $\widetilde{\mathcal Z}$ is induced by $\mathcal Z$.
\end{definition}
By Proposition \ref{prop:pullbackaffine}(4),
if $\widetilde{\mathcal Z}$ is induced by a projective structure on $C$, there is exactly
one projective structure on $C$ that induces $\widetilde{\mathcal Z}$.

\begin{proposition}\label{prop:pullbackaffine2}
Let $\pi\,:\, \widetilde{C}\, \longrightarrow\, C$ be an \'etale Galois covering
whose Galois group is $G$. 
\begin{enumerate}
\item 
Fix $\widetilde{\mathcal Z} \,\in\, \mathcal P(\widetilde C)^G$ (defined
in \eqref{projCG}). 
Consider the atlas of $C$ formed by the charts 
\begin{equation*}
z \circ (\pi\big\vert_{V})^{-1} \,:\, U \,\,\longrightarrow\, \mathbb{C}
\end{equation*}
where $U$ is an evenly covered open subset of $C$ with $V$ being a connected component of $\pi^{-1}(U)$ and $z \,:\, V \,\longrightarrow\, \mathbb{C}$ is a chart in $\widetilde{\mathcal Z}$. Then this is a projective atlas on $C$. 
\item 
Denote by $\mathcal Z$ the projective structure on $C$ defined by the atlas described in (1). 
Then ${\mathcal Z} \,\in\, \mathcal P( C)$ is the unique projective structure such that 
$\pi^*\mathcal Z \,=\, \widetilde{\mathcal Z}$. In other words,
$\mathcal Z $ is the unique projective structure such that
$\pi\,:\, (\widetilde{C},\, \widetilde{\mathcal Z})\,\longrightarrow\, (C,\,\mathcal Z)$
is projective. 
\item $\pi^\ast \mathcal{P}(C)\,=\, \mathcal P (\widetilde{C})^G $
(see Proposition \ref{prop:pullbackaffine}(1) for notation).
\end{enumerate}
\end{proposition}

The proof of Proposition \ref{prop:pullbackaffine2} is immediate.

\subsection{Prym varieties and the projective structure on the cover}
\label{subs:prym}

Consider an \'etale covering of degree two
\begin{equation}\label{epi}
\pi\,:\, \widetilde{C} \,\longrightarrow\, C
\end{equation}
between compact Riemann surfaces, where $g\,=\, \text{genus}(C)$; so
$\widetilde{g}\,=\, \text{genus}(\widetilde{C}) \,=\, 2g-1$. 
Starting from \eqref{epi} one can construct a Prym variety and two projective structures, one on $\widetilde C$, the other on $C$. 
The modern study of Prym varieties was initiated in \cite{Mu}. See also \cite{Be} for a survey.
Denote by 
\begin{equation}\label{esig}
\sigma\ :\ \widetilde{C}\ \longrightarrow\ \widetilde{C}
\end{equation}
the nontrivial deck transformation for $\pi$. For brevity, set
\begin{gather*}
 \widetilde S\ :=\ \widetilde C \times \widetilde C
\end{gather*}
and denote by 
\begin{equation}\label{eS}
\Sigma
\,:=\, \{(z\,, \sigma(z))\, \in\, \widetilde{S}\,\, \big\vert\,\,
z\, \in\, \widetilde{C}\} \,\subset\, \widetilde{S}
\end{equation}
the graph of $\sigma$ in \eqref{esig}.
Let $\widetilde{\Delta}\, \subset\, \widetilde{S}$ be the reduced diagonal. 
Notice that $\Sigma \cap \widetilde \Delta \ =\ \emptyset$.
Set
\begin{gather}\label{deftbb}
\tbb\ :=\ K_{\widetilde{S} } (2\widetilde{\Delta} )\ \lra\ \widetilde{S}.
\end{gather}
We recall some fundamental facts about Prym varieties.
The map $\pi$ induces the norm map on the Picard groups
$$Nm_\pi\ :\ {\rm Pic}^d(\widetilde{C})\ \longrightarrow\ {\rm Pic}^d(C).$$
The kernel of $Nm_\pi$ has two connected components that we will call $P$ and $P'$. The connected component containing the origin defines the Prym Variety associated with the covering $\pi$:
$$P\ := \ (\ker Nm_\pi)^0.$$
The components of $\ker Nm_\pi$ can be characterised as follows:
\begin{equation}
\label{eq:peven}
P\ =\ \left\{\mathcal{O}_{\widetilde C}\left(\sum_{i=1}^n p_i-\sigma(p_i)\right)\,\,
\big\vert\,\, n\,\equiv\, 0 \mod \,2\right\}
\end{equation}
\begin{equation}
\label{eq:podd}
P'\ =\ \left\{\mathcal{O}_{\widetilde C}\left(\sum_{i=1}^n p_i-\sigma(p_i)\right)
\,\, \big\vert\,\, n\,\equiv\, 1 \mod \,2\right\}. 
\end{equation}
The preimage of $K_C$, under the map $Nm_\pi$, also has two connected components
$$Nm_\pi^{-1}(K_{C})\ =\ P^+\cup P^- ,$$
where
\begin{equation*}
\begin{split}
P^+\ & =\ \{L\,\in\, Nm_\pi^{-1}(K_C)\,\,\big\vert\,\, h^0(L)\,\equiv\, 0\mod \, 2\}, 
\\
P^-\ & =\ \{L\,\in\, Nm_\pi^{-1}(K_C)\,\,\big\vert\,\, h^0(L) \,\equiv\, 1 \mod \,2\}.    \end{split}
\end{equation*}
Consider the component $P'\, \subset\, \ker Nm_\pi$ in \eqref{eq:podd}. Define the Abel--Prym map
$$\iota\ :\ \widetilde{C}\ \longrightarrow\ P', \ \ \, p\ \longmapsto\ \mathcal{O}_{\widetilde C}(p-\sigma (p))$$
and the Prym--difference map 
\begin{equation}\label{epd}
\phi\ :\ \widetilde{S}\ \longrightarrow \ P, \ \ \, (p,\,q)\ \longmapsto
\ \mathcal{O}_{\widetilde C}(p-\sigma(p) -q+\sigma (q)).
\end{equation}
The map $\iota$ extends by linearity to a map (which we will denote by the same symbol)
$$\iota\ :\ \operatorname{Pic}(\widetilde C)\ \longrightarrow\ P\cup P'\ \subset\ 
\operatorname{Pic}^0(\widetilde C)$$
it follows from \eqref{eq:peven}, \eqref{eq:podd} that
\begin{equation}\label{l1}
P\ =\ \iota(\operatorname{Pic}^0(\widetilde C)), \ \ \, P'\ =\ \iota(\operatorname{Pic}^1(\widetilde C)).
\end{equation}
In $P^+$ there is the canonical Theta--Prym divisor:
$$\Xi^+\ \, :=\ \, \{L\,\in\, P^+\,\, \big\vert\,\, h^0(L)\,\neq\, 0\}.$$
We can choose (see for example \cite[p.~382]{BL}) a theta characteristic $\kappa^{\otimes 2}\,\cong\, K_{C}$ such that 
\begin{equation}
\label{eq:kappatranslate}
h^0(\pi^\ast\kappa)\ =\ 0\, \mod \, 2,
\end{equation}
where $\pi$ is the map in \eqref{epi};
notice that $\widetilde{\kappa}\,=\, \pi^\ast\kappa$ is a theta characteristic on $\widetilde C$
because $\pi$ is \'etale. Furthermore, $\sigma^\ast \widetilde{\kappa}\,=\,\widetilde{\kappa}$ and $\widetilde{\kappa}\,\in\, P^+$. From \eqref{eq:kappatranslate} and
the connectedness of $P$ it follows that $L\otimes\widetilde{\kappa}\,\in\, P^+$ for all
$L\,\in\, P$. Define the noncanonical theta--Prym divisor
\begin{equation}\label{ptd}
\Xi_{\widetilde\kappa}^+\ =\ \{L\otimes\widetilde{\kappa}^{-1}\,\,\big\vert\,\,L\,\in\, \Xi^+\}\ \subset\ P.
\end{equation}
The involution $\sigma$ in \eqref{esig} induces involutions
$$\sigma^\ast\, :\, {\rm Pic}^0(\widetilde{C})\,\longrightarrow\, {\rm Pic}^0(\widetilde{C}), \ \ \, \sigma^\ast\,:\,{\rm Pic}^{g-1}(\widetilde{C})\,\longrightarrow\, {\rm Pic}^{g-1}(\widetilde{C})$$
defined by $\xi\, \longmapsto\, \sigma^*\xi$. Note that
$\sigma^\ast (\Xi^+)\,=\,\Xi^+$; this and the above observation that $\sigma^\ast \widetilde\kappa\,=\,\widetilde\kappa$
combine together to give that $\sigma^\ast (\Xi_{\widetilde\kappa}^+)\, =
\, \Xi_{\widetilde\kappa}^+$.

The restriction of $\sigma^\ast$ to $P\,=\,\iota({\rm Pic}^0(\widetilde{C}))$
(see \eqref{l1}) coincides with $-1_P$, so
\begin{equation}\label{a2}
2\Xi_{\widetilde\kappa}^+\ =\ \Xi_{\widetilde\kappa}^+ +\sigma^\ast
\Xi_{\widetilde\kappa}^+\ =\ \Xi_{\widetilde\kappa}^+ +(-1_P)^\ast \Xi_{\widetilde\kappa}^+ ,
\end{equation}
where $\Xi_{\widetilde\kappa}^+$ is is constructed in \eqref{ptd}.
The linear equivalence class of the divisor in \eqref{a2} is independent of $\kappa$. 
Indeed, if $\kappa_1,\, \kappa_2$ are two theta characteristics on $C$ satisfying \eqref{eq:kappatranslate}, then $\Xi_{\pi^\ast\kappa_1}^+$ and $\Xi_{\pi^\ast\kappa_2}^+$ differ by a translation, and hence $\Xi_{\pi^\ast\kappa_1}^+ +(-1_P)^\ast \Xi_{\pi^\ast\kappa_1}^+$ and $\Xi_{\pi^\ast\kappa_2}^+ +(-1_P)^\ast \Xi_{\pi^\ast\kappa_2}^+$ are linearly equivalent by the theorem of the square. Therefore, from \eqref{a2} it follows that the line bundle on $P$
\begin{equation*}
{\mathcal O}_P(2\Xi)\ \,=\ \, {\mathcal O}_P (2\Xi_{\pi^\ast\kappa})
\end{equation*}
is well defined.

The next result plays a crucial role in the paper.

\begin{theorem}\label{prop:fundamentalpullback}
There is an isomorphism of line bundles
\begin{equation*}
\phi^\ast {\mathcal O}_P\left(2\Xi\right)\ \cong\ K_{\widetilde{S}}\otimes
 \mathcal{O}_{\widetilde{S}} (2\widetilde{\Delta}-2\Sigma)\ =\ \tbb(-2\Sigma),
\end{equation*}
where $\phi$ is the Prym--difference map in \eqref{epd}, $\Sigma$ is defined in
\eqref{eS}, while $\tbb$ and $\widetilde\Delta$ are as in \eqref{deftbb}.
\end{theorem}

\begin{proof}
Choose $\kappa$ as in \eqref{eq:kappatranslate}.
For $q\,\in\, \widetilde C$, let 
$$\epsilon_q\ :\ C\ \lra\ P, \ \ \ x\ \longmapsto\ x-\sigma(x)+q-\sigma (q).$$ For all $L\in P^+$, we have the Jacobi inversion theorem for Prym varieties (see for example \cite[Lemma 3.2]{Na})
$$\epsilon_q^*(\Xi^+_{L})\ \cong\ \sigma ^*(L)\otimes \mathcal O_{\widetilde C}(\sigma (q)-q).$$
Then 
$$\phi^\ast (\Xi_{\pi^\ast\kappa}^+)\big\vert_{\widetilde{C}\times \{q\}}\,\cong\, \epsilon^\ast_{\sigma(q)}(\Xi_{\pi^\ast\kappa}^+)\cong \sigma^\ast(\pi^\ast\kappa)\otimes \mathcal{O}_{\widetilde C}(q-\sigma(q))\,=\,\pi^\ast\kappa\otimes \mathcal{O}_{\widetilde C}(q-\sigma(q))$$
and 
$$\phi^\ast (\Xi_{\pi^\ast\kappa}^+)\big\vert_{\{p\}\times \widetilde{C}}\ \cong\ \sigma^\ast \epsilon^\ast_{p}(\Xi_{\pi^\ast\kappa}^+)\ \cong\ \pi^\ast\kappa\otimes \mathcal{O}_{\widetilde C}(p-\sigma(p)).$$
Now by the see-saw theorem,
$$\phi ^*(\Xi^+_{\pi^\ast\kappa})\ \cong\ p_1^*(\pi^\ast\kappa) \otimes p_2^* (\pi^\ast\kappa) \otimes \mathcal{O}_{\widetilde{S}}(\widetilde{\Delta} - \Sigma).$$
Taking the square  completes the proof.
\end{proof}

Since $\Xi_{\pi^\ast\kappa}$ is a principal polarization on $P$, the divisor $2\Xi$ is 
base--point free by the theorem of the square. We note that
$H^0(P,\, 2\Xi)$ is endowed with a natural conformal class of Hermitian structures,
i.e., a class of Hermitian structures differing only by multiplication by a positive scalar. This natural conformal class is discussed, for example, in \cite[2.5]{IG} and with a notation  similar to that used in the present paper in see \cite[Section 2]{BGV}.
Since the orthogonal complement is the same for all Hermitian structures in a conformal class, the line 
\begin{equation}
\label{eq:defines}
\mathscr L
:=\ \left\{w\,\in\, H^0(P,\, 2\Xi)\,\,\big\vert\,\, w(0)\,=\,0\right\}^\perp\,\subset\, H^0(P,\, 2\Xi)
\end{equation}
is intrinsically defined. 
Let $s$ be a generator of this line.
From Theorem \ref{prop:fundamentalpullback}
it follows that $\phi^\ast s$ is a holomorphic section of $H^0(\widetilde{C}\times
\widetilde{C}, \, K_{\widetilde{C}\times \widetilde{C}}(2\widetilde{\Delta}-2\Sigma))$, which is again well defined up to multiplication by a nonzero scalar. Now $\phi(\widetilde{\Delta})\,=\, 0$, and $s(0)\,\neq \,0$, and hence $(\phi^\ast s)\big\vert_{\widetilde\Delta}\, \neq\, 0$.

Since $\phi^*s \, \in\, H^0(\widetilde{S},\, \tbb(-2\Sigma))\, \subset\,
H^0(\widetilde{S},\, \tbb)$ (see \eqref{deftbb} 
for the definition of $\tbb$), it follows from \cite[p.~756, Proposition 2.10]{BR1}
 that 
 \begin{equation*}
     \frac{1}{(\phi^\ast s)\big\vert_{\widetilde\Delta}}
(\phi^\ast s)\big\vert_{2\widetilde\Delta}
 \end{equation*} coincides with the canonical trivialisation $s_0$ of $\tbb\big\vert_{2\widetilde \Delta}$ (see \eqref{a1}).
Therefore, the section
\begin{equation*}
\frac{1}{(\phi^\ast s)\big\vert_{\widetilde\Delta}}(\phi^\ast s)\big\vert_{3\widetilde\Delta}\ \,\in\ \, \mathcal{V}_{\widetilde C}
\end{equation*}
gives a projective structure
\begin{equation}\label{cps}
\beta_\pi^Q\ =\ \beta^Q_{\pi:\widetilde{C}\to C} \  \ \in\ {\mathcal P}(\widetilde{C}).
\end{equation}
(See \eqref{evc} for the definition of $\mathcal{V}_{\widetilde{C}}$.)

\subsection{The projective structure on base curve}
\label{proonbase}

The goal of this subsection is to show that the Prym projective structure $\beta_{\pi}^Q$ on $\widetilde{C}$ is induced by a projective structure on $C$ in the sense of Definition \ref{definduced}.
In particular, this gives rise to a projective structure on $C$.

We start by collecting some basic facts about the pullback map
\begin{equation}\label{ephd}
\phi^*\ :\ H^0(P, \, 2\Xi)\ \lra\ H^0(\widetilde{S},\,\phi^*2 \Xi),
\end{equation}
where $\phi$ is the map in \eqref{epd}.
Consider the maps
\begin{gather*}
\rho\ :\ \widetilde{S} \longrightarrow\ P, \ \ \, \rho(p,\,q)
:= \mathcal{O}_{\widetilde C}(p-\sigma (p) + q-\sigma (q)),\\
j\ :\ \widetilde{S}\ \longrightarrow\ \widetilde{S}, \ \ \, j(p,\,q) :=(p,\,\sigma (q)).     
\end{gather*}
We have $\rho\,=\,\phi\circ j$ and $\phi\,=\,\rho\circ j$.
This 
makes the study of $\phi$ and $\rho$ equivalent. The map 
$$\rho^\ast\ :\ H^0(P,\, 2\Xi)\ \longrightarrow\ H^0(\widetilde{S},
\, K_{\widetilde{S}}(-2\widetilde{\Delta}+2\Sigma))$$
has been studied in \cite{IZ, IP}. We will give a brief overview of the results
on $\rho^\ast$ that will be used in the following.

By the projection formula, or the K\"unneth formula, we have the following identifications
\begin{equation}\label{ek}
H^0(\widetilde{S},\, K_{\widetilde{S}})\ \cong \ H^0(\widetilde{C},\, K_{\widetilde C})^{\otimes 2} \ \cong \ H^0(\widetilde S, \, p_1^\ast K_{\widetilde C}\otimes p_2^\ast K_{\widetilde C}).
\end{equation}
On this space we have two natural involutions: one induced by
$\textbf{t}$ in \eqref{et} and the other induced by the involution $\sigma\times \sigma$
of $\widetilde{S}$.
We let
\begin{equation}\label{bif}
\operatorname{Sym}^2 H^0(\widetilde{C}, \,K_{\widetilde C})^+\ :=\ \{\omega\,\in\, H^0(\widetilde S, \, p_1^\ast K_{\widetilde C}\otimes p_2^\ast K_{\widetilde C})\,\,\big\vert\,\, (\sigma\times \sigma)^\ast\omega\,=\, \omega\,=\, t^\ast\omega\}
\end{equation}
be the subspace of elements fixed by both involutions.
Notice that the identification in \eqref{ek} has been used.
\begin{remark}
\label{remark:danger}
The global sections of $p_1^*K_{\widetilde C} \otimes p_2^* K_{\widetilde C} $ that are invariant by $\textbf t$ correspond to $\operatorname{Sym}^2 H^0(K_{\widetilde C})$.
However, notice that the isomorphism $H^0(K_{\widetilde{C}\times \widetilde{C}})\,\cong\, H^0(K_{\widetilde C})^{\otimes 2}$ is not $\textbf t$--equivariant. In fact, 
the $\textbf t$--invariant sections of $H^0(p_1^\ast K_{\widetilde C}\otimes p_2^\ast K_{\widetilde C})$ correspond to the $\textbf t$--anti-invariant sections of $H^0(K_{\widetilde{S}})$.
\end{remark}

Let $$\pi\times \pi\,:\,\widetilde{S}\,\longrightarrow\, S, \ \ \, (p,\, q)\,\longmapsto\, (\pi(p),\,\pi(q)),$$
where $\pi$ is the map in \eqref{epi}. The pullback through $\pi\times \pi$ gives  
 maps (see \eqref{bif})
$$\operatorname{Sym}^2H^0(C, \,K_C)\ \hookrightarrow\ \operatorname{Sym}^2 H^0(\widetilde{C}, \, K_{\widetilde C})^+, $$
$$\operatorname{Sym}^2H^0(C,\, K_C)\cap H^0(S,\, K_{S}(-2\Delta))\,\hookrightarrow\, \operatorname{Sym}^2 H^0(\widetilde{C},\, K_{\widetilde C})^+\cap H^0(\widetilde{S},\, K_{\widetilde{S}}(-2\widetilde{\Delta}-2\Sigma)).$$

The following is a translation of Corollary 4.5 and Corollary 4.6 of \cite{IZ} in terms of $\phi$.

\begin{definition}
Let $\Gamma_0\,\subset\, H^0(P,\, 2\Xi)$ denote the space of sections that vanish at the origin and let  $\Gamma_{00}\,\subset\, H^0(P,\, 2\Xi)$ denote the space of  sections that vanish, at the origin, of order at least four. 
\end{definition}
In the following Proposition    $\phi^\ast$ denotes the map in \eqref{ephd} composed with the isomorphism of Theorem \ref{prop:fundamentalpullback}.
\begin{proposition}[{\cite[Corollary 4.5, Corollary 4.6]{IZ}}]\label{prop:fromizadi}
\phantom{p}
 \begin{enumerate}
\item $\phi^*$ maps $\Gamma_0$  onto the subspace
$$\operatorname{Sym}^2 H^0(\widetilde{C},\, K_{\widetilde C})^+\cap H^0(\widetilde{S}, \, K_{\widetilde{S}}(-2\Sigma))$$
of $H^0(\tbb (-2\Sigma))$.
\item $\phi^\ast$ maps $\Gamma_{00}$ onto the image of $$\operatorname{Sym}^2
H^0(C, \, K_C)\cap H^0(S,\, K_{S}(-2\Delta))$$
via the pullback map 
$\pi^*\otimes \pi^*\, :\, H^0(C,\, K_C)^{\otimes 2} \lra 
H^0( \widetilde{C},K_{\widetilde{C}})^{\otimes 2}$.
\end{enumerate}
\end{proposition}

The following fact can be deduced from the arguments used in Section 2 of \cite{IZ}, but it is not stated explicitly there. Therefore we give a proof.

\begin{proposition}
\label{prop:boa}
$\phi^\ast $ maps $H^0(P, \, 2\Xi)$ to the $\sigma\times \sigma$ -- invariant part of $H^0(\widetilde S, \, K_{\widetilde S}(2\widetilde S-2\Sigma))$.
\end{proposition}
\begin{proof}

From the definition of $\phi$ in \eqref{epd} we see that $\phi(\sigma p, \sigma q)=-\phi(p,q)$.
The statement follows because $\Xi$ is symmetric and all sections of $2\Xi$ are invariant with respect to the involution $(-1_P)$ of $P$ (see \cite[4.6]{BL}).
\end{proof}

Now we are ready to show that the Prym projective structure on $\widetilde{C}$ 
in \eqref{cps} is induced by a projective structure on $C$.

\begin{corollary}
\label{cor:descent}
The projective structure $\beta_\pi^Q\,\in\, \mathcal{P}(\widetilde{C})$
in \eqref{cps} is induced by a projective structure $$\beta^P_\pi \ =\ \beta^P_{\pi:\widetilde{C}\to C}\,\in\,\mathcal{P}(C).$$
\end{corollary}

\begin{proof}
In view of Proposition \ref{prop:pullbackaffine2}, we need to prove that $\sigma\cdot
\beta_\pi^Q\,=\,\beta_\pi^Q$, and by Lemma \ref{Fequivanriant} this is equivalent to proving that $\sigma\cdot F(\beta_\pi^Q)\,=\,F(\beta_\pi^Q)$, where $F$ is the map constructed in 
\eqref{eq:Fmap}; see Lemma \ref{Fequivanriant} for the actions of $\sigma$ on $\mathcal{P}(\widetilde C)$ and $\mathcal{V}_{\widetilde C}$.
Recall that $\beta_\pi^Q$ is the unique projective structure such that
\begin{equation*}
F(\beta_\pi^Q)\,=\,\frac{1}{\phi^\ast s\big\vert_{\widetilde\Delta}}\phi^\ast
s\big\vert_{3\widetilde\Delta},
\end{equation*}
 where $s$ is a generator of the line $\mathscr L$ constructed in \eqref{eq:defines}. 
By \ref{prop:boa},
$\phi^\ast$ maps $H^0(P,\, 2\Xi)$ to the $\sigma\times \sigma$--invariant part of $H^0(\widetilde{S},\, K_{\widetilde{S}}(2\widetilde{\Delta}-2\Sigma))$. Therefore,
\begin{align*}
\sigma\cdot F(\beta_\pi^Q)\ &=\ \sigma\cdot\left(\frac{1}{\phi^\ast s|_{\widetilde\Delta}}\phi^\ast s|_{3\widetilde\Delta}\right)\ =\ \left.\left(\frac{1}{\phi^\ast s|_{\widetilde{\Delta}}}(\sigma\times \sigma)^\ast \phi^\ast s\right)\right|_{3\widetilde{\Delta}}\\
&=\ \left.\left(\frac{1}{\phi^\ast s|_{\widetilde{\Delta}}} \phi^\ast s\right)\right|_{3\widetilde \Delta}\ =\ \frac{1}{\phi^\ast s|_{\widetilde\Delta}}\phi^\ast s|_{3\widetilde\Delta}\ =\ F(\beta_\pi^Q).
\end{align*}
By Lemma \ref{Fequivanriant} the map $F$ is equivariant, so $\beta^Q_\pi$ is $\sigma$--invariant. Proposition \ref{prop:pullbackaffine2} (1) gives the result. 
\end{proof}

\section{Relative setting}

This section prepares for the study of the $\debar$--derivative of the section $\beta^P$ that will be the object of the next section. We first recall some well-known facts on Prym varieties, then the definition and some basic properties of Torelli space. We then compute explicitly the codifferential of the Prym period map. Finally, we recall  from \cite{BCFP} the formalism allowing to consider a family of varying projective structures on varying curves as a section of an affine bundle over the moduli space of curves. In our case, it  is convenient to perform this construction at the level of the Torelli space.

\subsection{Preliminaries on Prym varieties}\label{subs:topologicalconstruction}

We will follow \cite{FR} with a slightly different notation.
Let $C$ be a compact Riemann surface of genus $g$, and let $\mathcal{B}\,=\,\{a_1,\,\cdots ,\, a_g,\, b_1,\,\cdots ,\, b_g\}$ be a (ordered) symplectic basis of $H_1(C,\,\mathbb{Z})$ with respect to the intersection product:
$$a_i\cdot a_j\,=\,0,\ \ \, a_i\cdot b_j\,=\,\delta_{ij},\ \ \, b_i\cdot b_j\,=\,0.$$
Fix the (unique) basis $\omega_1,\,\cdots ,\, \omega_g$ of $H^0(C,\,K_C)$ satisfying
\begin{equation}\label{db}
\int_{a_i}\omega_j\ =\ \delta_{ij},
\end{equation}
and define the period matrix $\Pi^{C}$ by
\begin{equation}\label{eq:periodmatrix}
\Pi^{C}_{ij}\ \,:=\ \,\int_{b_i}\omega_j.
\end{equation}
Then  $\Pi^C$ lies in the Siegel upper-half space $\mathbb{H}_g$, which is the set of
$g\times g$ complex symmetric matrices with positive definite imaginary
part.

Let 
$\pi \,:\,\widetilde{C}\,
\longrightarrow\, C$ 
be the \'etale double cover such that  $\Im \pi_*   $ is the subgroup of $H_1(C,\Z)$ (or equivalently of the fundamental group) generated by $\{a_1, 2b_1, a_i, b_i \}_{ i=2,\ldots,g}$. Let  $\widetilde{\mathcal{B}}\,=\, \{\widehat{a}_1,\,\cdots
,\, \widehat{a}_{2g-1},\, \widehat{b}_1,\, \cdots ,\, \widehat{b}_{2g-1}\}$ be 
 the symplectic basis of $H_1(\widetilde C, \,\mathbb{Z})$ such that
\begin{equation}\label{dabi}
\begin{cases}
\pi_\ast(\widehat{a}_1)\ =\ a_1 \\
\pi_\ast({b}_1)\ =\ 2 b_1\\
\pi_\ast (\widehat{a}_i)\,=\, \pi_\ast(\widehat{a}_{i+g-1})\,=\,a_i, & i\,=\,2,\,\cdots ,\,g\\
\pi_\ast (\widehat{b}_i)\,=\,\pi_\ast(\widehat{b}_{i+g-1})\,=\,b_i, & i\,=\,2,\,\cdots ,\,g.
\end{cases}
\end{equation}
Denote by $\{\widehat{\omega}_1,\, \cdots ,\, \widehat{\omega}_{2g-1}\}$ the normalized basis of $H^0(\widetilde{C},\, K_{\widetilde{C}})$ relative to the marking $\widetilde{\mathcal{B}}$, so
\begin{equation}\label{dwti}
\int_{\widehat a_i}\widehat\omega_j\ =\ \delta_{ij}.
\end{equation}
Then the Galois involution $\sigma\,:\, \widetilde{C}\, \longrightarrow\, \widetilde{C}$
for $\pi$ satisfies the following:
\begin{equation}
\begin{cases}
\sigma^\ast\widehat\omega_1\ =\ \widehat \omega_1\\
\sigma^\ast\widehat\omega_i\ =\ \widehat\omega_{i+g-1},& i\,=\,2,\,\cdots ,\,g.
\end{cases}
\end{equation}
For the basis $\omega_1,\,\cdots ,\, \omega_g$ in \eqref{db},
$$\begin{cases}
\pi^\ast\omega_1\ =\ \widehat \omega_1\\
\pi^\ast\omega_i\ =\ \widehat\omega_i+\widehat\omega_{i+g-1},& i\,=\,2,\,\cdots ,\,g.
\end{cases}$$
Sometimes $\pi^*\omega_i$ will be identified with $\omega_i$ to simplify the  notation. 

 The holomorphic 
differential forms
\begin{equation*}
\gamma_{i-1}\ \,:=\ \, \widehat \omega_i-\widehat \omega_{i+g-1},\qquad i\,=\,2,\,\cdots ,\,g,
\end{equation*}
are called \emph{normal Prym differentials} with respect to $\mathcal{B}$.
These differentials satisfy the following:
$$\int_{\widehat a_{i+1}} \gamma_j\ =\ \delta_{ij},\qquad 1\,\leq\, i,\,j\,\leq\, g-1 .$$
The period matrix $\tau^C\,\in\, \mathbb{H}_{g-1}$ given by
\begin{equation}
    \label{deftauc}
\tau^C_{ij}\ :=\ \int_{\widehat b_{i+1}} \gamma_j, \qquad 1\,\leq \,i,\,j\,\leq\, g-1
\end{equation}
is the period matrix of the Prym Variety associated to the covering $\pi$, that is
\begin{equation}
\label{eq:prymtorus}
(P(\widetilde{C}\to C),\,\Xi)\ \cong\ \left(\frac{\mathbb{C}^{g-1}}{\mathbb{Z}^{g-1}\oplus \tau^C\mathbb{Z}^{g-1}},\,\Theta_{\tau^C}\right),
\end{equation}
where $\Theta_{\tau^C}$ is the divisor of the Riemann Theta Function with period matrix $\tau^C$:
\begin{equation}
\theta(\tau^C;\,z)\ :=\ \sum_{m\in \mathbb{Z}^{g-1}} e^{\pi \sqrt{-1}m^t\tau^C m+2\pi \sqrt{-1} m^tz}.
\end{equation}
Using the identification in \eqref{eq:prymtorus}, the Prym difference map $\phi$ in
\eqref{epd} lifts to 
a map from the
universal cover of $\widetilde{S}$ to  $\mathbb{C}^{g-1}$, the universal cover of $P$, given by the following formula:
\begin{equation}
\label{Prym-diff}
\phi(p,\, q)\ = \ \left(\int_{q}^p \gamma_i\right)_{i=1,\cdots ,g-1}.
\end{equation}
For $\tau \in \mathbb H_{g-1}$
consider  the  space of second order theta functions:
\begin{equation}
        \label{eq:secondorder}
\begin{split}
    V(\tau)\,:= \{&f \in \mathcal {O}( \mathbb{C}^{g-1}) \, \,\, \big\vert\,\,
\\
&
f(z+n+\tau m)\,=\, \exp({-2\pi\sqrt{-1} ( m^t\tau m+2m^t z})f(z)\,\,\\
&\forall\,\,\, n,\,m\,\in\, \mathbb{Z}^{g-1}\}.
\end{split}
\end{equation}
Then we have an identification
\begin{gather*}
 H^0(P,\, 2\Xi)\ \, \xrightarrow{\,\,\, \cong\,\,\,}\ \, V(\tau^C),
\end{gather*}
see e.g. \cite[Section 4.1]{VG}.

\subsection{The Torelli space}
\label{subs:family1}

Consider the Torelli space $\widetilde M_g$ parametrizing  
pairs $(C,\mathcal B)$, where $C$ is a smooth projective curve  of genus $g$  and $\mathcal B$ is a symplectic basis of $H_1(X,\, \mathbb{Z})$; see \cite[p.~460]{ACGH2}. The moduli space
$\widetilde M_g$ is acted upon by $\operatorname{Sp}(2g,\mathbb{Z})$, and the quotient $\widetilde M_g/\operatorname{Sp}(2g,\mathbb{Z})$ is  $M_g$. The period map
\begin{equation}
\label{defPi}
\Pi\ :\ \widetilde M_g\ \longrightarrow\ \mathbb{H}_g
\end{equation}
sends a pair $(C,\,\mathcal{B})$ to the period matrix of $C$ with respect to $\mathcal{B}$
(see \eqref{eq:periodmatrix}). This map lifts the Torelli map $j$, that is, the following diagram is commutative:
$$\begin{tikzcd}
\widetilde M_g \arrow[r, "\Pi"] \arrow[d] & \mathbb{H}_g \arrow[d] \\
M_g \arrow[r, "j"] & A_g 
\end{tikzcd}$$
The map $\Pi$ is ${\rm Sp}(2g,\mathbb{Z})$--equivariant. The moduli space
$\widetilde M_g$ is smooth and carries a universal family of curves
\begin{equation}\label{eqf}
f\ :\ \mathcal{C}\ \longrightarrow\ \widetilde{M}_g.
\end{equation}

We will use the notation $K_{\mathcal{C}/\widetilde M_g}$ for the relative canonical bundle.
It is well-known that  $f_\ast K_{\mathcal{C}/\widetilde M_g}^{\otimes 2}\, \cong \, \Omega^1_{\widetilde M_g}$. We need to write this isomorphism explicitly.
The dual of the Kodaira--Spencer map $\rho_b\,:\,T_b \widetilde M_g\,\stackrel{\cong}{\longrightarrow}\, H^1(C_b, \, T_{C_b})$ gives an isomorphism
\begin{equation*}
\rho^*\, :\, (R^1f_* T_f)^* \, \stackrel{\cong}{\longrightarrow}\,\Omega^1_{\widetilde M_g},    
\end{equation*}
where $T_f \, \longrightarrow\, {\mathcal C}$ is the relative tangent bundle. The Serre duality pairing
\begin{equation}
\label{eq:serre}
H^1(C_b,\, T_{C_b})\otimes H^0(C_b,\, 2K_{C_b})\ \longrightarrow\ \mathbb{C},
\ \ \, \xi\otimes \omega\ \longmapsto\ \int_{C_b} \xi \cup \omega
\end{equation}
yields an isomorphism
$$S\, :\,f_\ast K_{\mathcal{C}/\widetilde M_g}^{\otimes 2}\,
\stackrel{\cong}{\longrightarrow}\,(R^1f_* T_f)^*.$$
The composition of the two gives the following isomorphism that we will use
\begin{equation}
\label{eq:lambdaiso}
\lambda\ :\ =\ \rho^* \circ S\ :\ f_\ast K_{\mathcal{C}/\widetilde M_g}^{\otimes 2}\ 
\longrightarrow\ \Omega^1_{\widetilde M_g}.
\end{equation}

Applying the construction in Section \ref{subs:topologicalconstruction} at every point $b=(C,\mathcal B) \in \widetilde{M}_g$ produces a family of étale covers
\begin{equation}\label{etf}
\widetilde{\mathcal{C}}\ \xrightarrow{\,\,\,\pi\,\,\,}\ \mathcal{C}\ \xrightarrow{\,\,\,f\,\,\,}\ {\widetilde M}_g, \ \ \, \widetilde f\ :=\ f\circ \pi.
\end{equation}
Let
\begin{equation}\label{es2}
\sigma\ :\ \widetilde{\mathcal{C}}\ \longrightarrow\ \widetilde{\mathcal{C}}
\end{equation}
be the fiber preserving involution that on each fiber restricts  to the nontrivial element of $\operatorname{Gal}(\pi_b)$. 
From Section \ref{subs:topologicalconstruction} we  get sections $\widehat{a}_1,\, \cdots, \, \widehat{a}_{2g-1},\, \widehat{b}_1,\, \cdots ,\, \widehat{b}_{2g-1}$ of $R^1
\widetilde{f}_{\ast} \mathbb{Z}$,  holomorphic sections
\begin{equation}\label{nega}
\gamma_1,\, \cdots ,\, \gamma_{g-1}\ \in\ H^0(\widetilde{M}_g,\, 
\widetilde{f}_\ast K_{\widetilde{\mathcal{C}}/\widetilde M_g}),
\end{equation}
and  a period map 
\begin{equation}
\label{eq:prymperiod}
\tau\ :\ \widetilde M_g\ \longrightarrow\ \mathbb{H}_{g-1},\ \ \, b
\ \longmapsto\ \left(\int_{\widehat{b}_{i+1}(b)}\gamma_j(b)\right)_{i,j=1,\cdots ,{g-1}},
\end{equation}
see \eqref{deftauc}.
We now recall the moduli space of Prym varieties, which will be crucial in the following.

\begin{definition} 
Denote by $R_g$ the moduli space of étale double covers $\pi: \widetilde{C}\lra C$ where $C$ has genus $g$.
Two such covers are equivalent if there is a fibre preserving isomorphism of the covers. It follows that the bases of the covers are also isomorphic. An element of $R_g$ is an equivalence class of covers. The map 
\begin{gather*}
R_g \ \lra\ M_g, \ \ \,
[\pi : \widetilde{C} \lra C]\ \longmapsto\ [C], 
\end{gather*}
is a finite cover of degree $2^{2g}-1$. The Prym map
\begin{equation*}
P_g\ :\ R_g\ \lra\ A_{g-1},\ \ \, [\pi:\widetilde{C}\to C]\ \longmapsto
\ [(P(\widetilde{C}\to C),\, \Xi)],
\end{equation*}
associates to the equivalence class of a double cover the isomorphism class of its Prym variety.
\end{definition}
By \eqref{eq:prymtorus}, the map $\tau$ is a lift of $P_g$, meaning the following diagram commutes:
\begin{equation}
\label{RPg}
\begin{tikzcd}
\widetilde M_{g} \arrow[r, "\tau"] \arrow[d] & \mathbb{H}_{g-1} \arrow[d] \\
R_g \arrow[r, "P_g"'] & A_{g-1} .
\end{tikzcd}
\end{equation}
Here the left vertical arrow is the map 
$b \, \longmapsto\, [\pi_b\,:\,\widetilde C_b\, \longrightarrow\, C_b]$.

We now compute explicitly  the codifferential of the period map $\tau$.
This codifferential is well understood and is usually presented in coordinate--free manner as, for example, in \cite[7.4]{Be2}. 
Still, we do this computation because we wish to obtain an explicit expression for the codifferential in the coordinates given by $Z_{ij}$ on $\mathbb{H}_{g-1}$, and using the identification $H^0(C,\, 2K_C)\cong H^1(C,\, T_C)^\ast$ induced by \eqref{eq:serre}.
We were not able to find such an expression in the literature.
The reason why this expression is useful for our purposes will be apparent in the proof of Theorem \ref{teo:debar}.

\begin{proposition}
\label{prop:codifferential}
Let $Z_{ij}$ with $1\, \leq\, i\,\leq\, j\, \leq\, g-1$ be the coordinates on $\mathbb{H}_{g-1}$. 
Fix $b\in \widetilde{M}_g$. 
The quadratic differential $-\gamma_i\gamma_j(b)\,\in\, H^0(\widetilde C_{b},\, 2 K_{\widetilde C_b})$ 
is $\sigma_b$--invariant. Identify it with the element of $ H^0(C_b,\, 2K_{C_b})$, which pulls back to it.
Then
$$\tau^\ast d Z_{ij}\ \,=\ \,\lambda(-\gamma_i\gamma_j).$$
where $\lambda\,:\, f_\ast K_{\mathcal{C}/\widetilde M_g}^{\otimes 2}\, \longrightarrow\, \Omega^1_{\widetilde M_g}$ is defined in \eqref{eq:lambdaiso}.
\end{proposition}

\begin{proof}
Set for simplicity $\widetilde{C}\,=\, \widetilde{C}_b$,\, 
$\widehat{a}_i\,=\, \widehat a_i(b),$
etc. 
We have the Kodaira --Spencer maps $\widetilde\rho\,:\, T_b \widetilde{M}_g\, \longrightarrow\, H^1(\widetilde{C},\, T_{\widetilde C})$ and $\rho\,:\,T_b\widetilde{M}_g\, \longrightarrow\, T_b H^1(C,\, T_C)$ for the families $\widetilde f$ and $f$ respectively. They are  interlinked by the pullback through the cover (recall that $\pi^*T_C=T_{\widetilde{C}}$):
\begin{equation}
\label{eq:twokodairas}
\begin{tikzcd}
 & H^1(T_C) \arrow[rd, "\pi^\ast"] & \\
T_b\widetilde M_g \arrow[rr, "\widetilde\rho"] \arrow[ru, "\rho"] & & H^1(T_{\widetilde C})
\end{tikzcd}
\end{equation}
Fix a tangent vector $v\,\in\, T_b\widetilde{M}_g$.
By definition, we have
$$\lambda(-\gamma_i\gamma_j)(v)\ =\ \int_C \rho(v)\cup(-\gamma_i\gamma_j).$$
So we need to prove that
\begin{equation}
\label{eq:derivativegoal}
\tau^\ast d Z_{ij}(v)\ \,=\ \,\int_C \rho(v)\cup (-\gamma_i\gamma_j).
\end{equation}
 We have
\begin{equation}
\label{eq:st1}
\frac{\partial}{\partial v} \tau_{ij}\,=\,\frac{\partial }{\partial v}\int_{\widehat{b}_{i+1}}\gamma_j\, =
\,
\int_{\widehat{b}_{i+1} } (\frac{\partial }{\partial v}\widehat\omega_{j+1}- \frac{\partial }{\partial v}\widehat\omega_{j+g}),
\end{equation}
where $\widehat{\omega}_k$ are defined before \eqref{dwti}.
Split $\frac{\partial }{\partial v}\widehat \omega_k$ into a sum of the $(0,\,1)$ part and
the $(1,\,0)$ part. We know that
\begin{equation}
\label{eq:zeroone}
\left(\frac{\partial }{\partial v}\widehat \omega_k\right)^{0,1}\ \ =
\ \ \widetilde\rho(v)\cup \widehat \omega_k
\end{equation}
\cite[p.~220--221]{ACGH2}.
A holomorphic form on $\widetilde{C}_b$ can be expressed as a linear combination of
$\widehat{\omega}_1,\, \cdots ,\, \widehat{\omega}_{2g-1}$, and the coefficients of
this linear combination are determined by integration against $\hat{a}_i$
using
\eqref{dwti}. In particular, we have
\begin{equation*}
\left(\frac{\partial }{\partial v}\widehat \omega_k\right)^{1,0}\ =\
\sum_{h=1}^{2g-1}\left(\int_{\widehat a_h} \left(\frac{\partial }{\partial v}\widehat \omega_k\right)^{1,0}\right)\widehat \omega_h.
\end{equation*}
Since $\int_{\widehat a_h(b)}\widehat \omega_k(b)\,=\,\delta_{hk}$ is constant,
and using \eqref{eq:zeroone}, we deduce that
\begin{align*}
\left(\frac{\partial }{\partial v}\widehat \omega_k\right)^{1,0}\ &=\ -\sum_{h=1}^{2g-1}\left(\int_{\widehat a_h} \left(\frac{\partial }{\partial v}\widehat \omega_k\right)^{0,1}\right)\widehat\omega_h\ =\
-\sum_{h=1}^{2g-1} \left(\int_{\widehat a_h} \widetilde\rho(v)\cup \widehat \omega_k\right)\widehat\omega_h,\\
\int_{\widehat{b}_{i}} \frac{\partial }{\partial v} \widehat{\omega}_{k}\ & =\
\int_{\widehat{b}_{i}}
\bigl ( \frac{\partial \widehat{\omega}_{k} }{\partial v}  \bigr)^{0,1} 
+
\int_{\widehat{b}_{i}}
\bigl ( \frac{\partial \widehat{\omega}_{k} }{\partial v}  \bigr)^{1,0} \ 
=
\int_{\widehat{b}_{i}} \widetilde{\rho}(v)\cup \widehat{\omega}_{k}-\sum_{h=1}^{2g-1}
\int_{\widehat{a}_h} \widetilde{\rho}(v)\cup \widehat{\omega}_{k}\int_{\widehat{b}_{i}}\widehat{\omega}_h\ =\\
&=\ \sum_{h=1}^{2g-1}\int_{\widehat{b}_{h}}\widetilde{\rho}(v)\cup \widehat{\omega}_{k}
\int_{\widehat{a}_h} \widehat{\omega}_{i}-\sum_{h=1}^{2g-1} \int_{\widehat{a}_h} \widetilde{\rho}(v)\cup
\widehat{\omega}_{k}\int_{\widehat{b}_{h}}\widehat{\omega}_{i}.
\end{align*}
Here we used that $\int_{\widehat{a}_h}\widehat{\omega}_i\,=\,\delta_{ih}$ and
$\int_{\widehat{b}_i}\widehat{\omega}_h\, =\, \int_{\widehat b_h}\widehat\omega_i$. 
Recall that for closed differential forms $\alpha, \beta  $
 on 
 $\widetilde{C}$ the following identity holds:
\begin{gather*}
    \sum_{h=1}^{2g-1}
\biggl (    \int_{\widehat{a}_{h}}\alpha
\int_{\widehat{b}_h} \beta-  
\int_{\widehat{a}_h} \beta\int_{\widehat{b}_{h}}\alpha
\biggr )
=\int_{\widetilde{C}}\alpha\wedge \beta.
\end{gather*}
This identity is classically known as Riemann reciprocity law (see for example \cite[p.~231]{GH}).
It is equivalent to the fact that the functionals of integration along the cycles 
$\widetilde{\alpha}_h$ and $\widetilde{b}_h$
form  a symplectic basis of the dual to $H^1(\widetilde{C}, \C)$.
Applying this identity with $\alpha = \widehat{\omega}_i$ and $\beta=\widetilde{\rho}(v)\cup \widehat{\omega}_k$  we get 
\begin{equation}
\label{eq:fullderivative}
\int_{\widehat b_i} \frac{\partial}{\partial v}\widehat \omega_k\ =\ \int_{\widetilde{C}} \widehat\omega_i\wedge(\widetilde\rho(v)\cup \widehat \omega_k)\ =\ \int_{\widetilde C}\widetilde\rho(v)\cup(-\widehat\omega_i\widehat\omega_k).
\end{equation}
Substituting this in \eqref{eq:st1} we get that
$$\frac{\partial}{\partial v}\tau_{ij}\ =\ \int_{\widetilde C} \widetilde\rho(v)\cup(-\widehat\omega_{i+1}\widehat\omega_{j+1}+\widehat\omega_{i+1}\widehat\omega_{j+g}) 
\ =\ \int_{\widetilde C} \widetilde\rho(v)\cup(-\widehat\omega_{i+1}\gamma_j).$$
Now $\widetilde\rho(v)$ is $\sigma$-invariant (because of \eqref{eq:twokodairas}) while $\sigma^\ast (\widehat \omega_{i+1}\gamma_{j})\,=\,-\widehat\omega_{i+g}\gamma_{j}$. Thus
\begin{gather*}
\int_{\widetilde C} \widetilde{\rho}(v)\cup(-\widehat{\omega}_{i+1}\gamma_{j})\ = 
\ \int_{\widetilde C} \widetilde\rho(v)\cup(\widehat\omega_{i+g}\gamma_{j}),\\
\frac{\partial}{\partial v}\tau_{ij}\, = \, \int_{\widetilde C} \widetilde\rho(v)\cup(-\widehat \omega_{i+1}\gamma_{j})\, =\,
\frac{1}{2}\int_{\widetilde C} \widetilde\rho(v)\cup(-\widehat \omega_{i+1}\gamma_{j})+\frac{1}{2}\int_{\widetilde C} \widetilde\rho(v)\cup(\widehat\omega_{i+g}\gamma_{j})\\
=\, \frac{1}{2}\int_{\widetilde C} \widetilde \rho(v) \cup (-\gamma_i\gamma_j)\,=\,\int_{C} \rho(v)\cup(-\gamma_i\gamma_j).
\end{gather*}

This proves \eqref{eq:derivativegoal}, and the proof of the proposition is completed.
\end{proof}

\subsection{Projective structures on marked Riemann surfaces}
\label{ss:projmarked}
Consider the families of complex surfaces: 
\begin{equation}\label{ePh}
\Phi\ :=\ f\times_{\widetilde{M}_g} f\ :\ \mathcal S := \mathcal{C}\times_{\widetilde{M}_g} \mathcal{C}\ \longrightarrow\ \widetilde{M}_g
\end{equation}
where $f$ is the family of curves in \eqref{eqf} (the subscript in $\times_{\widetilde M_g}$ will often be omitted). 
For $k\,=\, 1,\, 2$,
$$
\mathbf{p}_k\ :\ \mathcal{S}\ \longrightarrow\
\mathcal{C}
$$
denotes the natural projection to the $k$--th factor. In $\mathcal{C}\times_{\widetilde M_g}\mathcal{C}$ there is the reduced diagonal divisor
$$\Delta_{\widetilde M_g}\ :=\ \{(x,\,x)\,\, \big\vert\,\, x\,\in\, \mathcal{C}\}
\ \subset\ \mathcal{S}.$$

The constructions done in Section \ref{subs:projectivestructures} work in a 
relative setting, which is briefly recalled (see \cite{BCFP}). 
Consider the line bundle
$$\mathbb{L}\ :=\ \mathbf{p}_1^\ast K_{\mathcal{C}/\widetilde M_g}\otimes \mathbf{p}_2^\ast K_{\mathcal{C}/\widetilde M_g}\otimes \mathcal{O}_{\mathcal{S}}(2\Delta_{\widetilde M_g})\ \longrightarrow\ \mathcal{S}.
$$
There is a canonical isomorphism
$$\mathbb{L}\big\vert_{\Delta_{\widetilde M_g}}\ \cong\ \mathcal{O}_{\Delta_{\widetilde M_g}}.$$
The holomorphic vector bundles $\mathcal V , \mathcal V_2$ on $\widetilde{M}_g$ are
defined as follows:
\begin{gather}
\label{defV}
\mathcal{V}\ :=\ \Phi_\ast \left(\mathbb{L}\otimes \frac{\mathcal{O}_{\mathcal{C}\times\mathcal{C}}}{\mathcal{O}_{\mathcal{S}}(-3\Delta_{\widetilde M_g})} \right),\qquad  \mathcal{V}_2\ :=\ \Phi_\ast \left(\mathbb{L}\otimes \frac{\mathcal{O}_{\mathcal{C}\times\mathcal{C}}}{\mathcal{O}_{\mathcal{S}}(-2\Delta_{\widetilde M_g})} \right),
\end{gather}
where $\Phi$ is the map in \eqref{ePh}.
The restriction from $3\Delta_{\widetilde M_g}$ to $2\Delta_{\widetilde M_g}$ gives a holomorphic morphism
\begin{gather}
 \label{psi:map}
 \Psi\ :\ \mathcal V\ \lra\ \mathcal V_2.
\end{gather}
This morphism $\Psi$ is surjective. The kernel of $\Psi$ is 
\begin{gather*}
 \ker \Psi\ =\ f_* K^{\otimes 2}_{\mathcal C /\widetilde{M}_g}
\end{gather*}
(see \cite{BCFP} for more details).
Consequently, we have a short exact sequence of holomorphic vector
bundles on $\widetilde M_g$:
\begin{gather}\label{ses:1}
0\ \lra\ f_*K_{\mathcal C / \widetilde{M}_g}^{\otimes 2}\ =\ 
\Omega^1_{\widetilde{M}_g} \ \lra\ \mathcal{V}\ \stackrel{\Psi}{\lra}\
{\mathcal V}_2\ \lra\ 0. 
\end{gather}
There is a holomorphic section $s_0 \,\in\, H^0(\widetilde{M}_g,\, {\mathcal V}_2)$ that sends any $b\,\in\, \widetilde{M}_g$ to the canonical section
constructed in \eqref{a1} for the Riemann surface $C_b$. Set
\begin{equation}\label{evs}
 \mathcal{V}_{s_0}\ :=\ \Psi^{-1} (s_0) \ \subset\ \mathcal{V},
\end{equation}
where $\Psi$ is the map in \eqref{psi:map}. This $\mathcal{V}_{s_0}$ is an affine bundle on $\widetilde M_g$ modelled on $f_* K_{\mathcal C /\widetilde M_g}^{\otimes 2}$. 
The fibre of $\mathcal{V}_{s_0}$ over the point $b$ is $\mathcal{V}_{C_b}$ constructed in \eqref{evc}. Therefore, smooth sections of $\mathcal{V}_{s_0}$ in \eqref{evs} correspond to
smooth families of projective structures on $\mathcal C$ because of the isomorphism $F$ in \eqref{eq:Fmap}.
For a smooth section $\beta$ of ${\mathcal V}_{s_0}$,
\begin{gather*}
 \Psi ( \debar \beta)\ =\ \debar \Psi \beta\ =\ \debar s_ 0\ =\ 0,
\end{gather*}
because $s_0$ is a holomorphic section of $\mathcal V_2$ and the homomorphism 
$\Psi$ is holomorphic. Consequently,
\begin{gather*}
\debar \beta\ \in\ \mathcal A^{0,1} ( \ker \Psi).
\end{gather*}
The above section $\debar \beta$ will be interpreted as a $(1,1)$ form on 
$\widetilde{M}_g$ (see \eqref{ses:1}).
The Prym projective structure 
gives rise to a section of $\mathcal{V}_{s_0}$ in the following way:
\begin{equation}
\label{eq:defbetap}
\beta^P\ :\ \widetilde M_g\ \lra\ \mathcal{V}_{s_0}, \ \ \, \qquad \beta^P(b) \ = \ \beta_{\pi_b};
\end{equation}
where   $\beta^P_{\pi_b}$
 is the projective structure constructed in  Corollary \ref{cor:descent}. 

For $k\,=\, 1,\,2$, let
\begin{equation}\label{wtpk}
\widetilde{\mathbf{p}}_k\ :\ \widetilde{S}:=\widetilde{\mathcal{C}}\times_{\widetilde M_g}\widetilde{\mathcal{C}}\ \longrightarrow\ \widetilde{\mathcal{C}}
\end{equation}
be the natural projection to the $k$--th factor (see \eqref{etf} for the definition of the family of curves $\widetilde f\,:\,\widetilde{\mathcal{C}}\,\longrightarrow\, \widetilde M_g$). Let
\begin{equation}\label{SiB}
\Sigma_{\widetilde{M}_g}\ :=\ \{(x,\, \sigma(x))\, \in\, \widetilde{\mathcal{C}}\times_{\widetilde M_g}\widetilde{\mathcal{C}}\,\, \big\vert\,\, x\,\in\, \widetilde C\}\ \subset\
\widetilde{S}
\end{equation}
be the relative graph of the involution $\sigma$ in \eqref{es2}. Let
\begin{equation*}
\widetilde \Delta_{\widetilde M_g}
\ :=\ \{(x,\, x)\, \in\, \widetilde{\mathcal{C}}\times_{\widetilde M_g}\widetilde{\mathcal{C}}\,\, \big\vert\,\, x\,\in\, \widetilde C\}\ \subset\
\widetilde{S}
\end{equation*}
be the reduced relative diagonal divisor. For each $b\, \in\, {\widetilde M}_g$, we will use the notation
$$
\widetilde{C}_b\ :=\ \widetilde{f}^{-1}(b), \ \ \, 
\widetilde \Delta_b\ :=\ \widetilde \Delta_{\widetilde M_g}\cap \left(\widetilde C_b\times \widetilde C_b\right),\ \ \, \Sigma_b\ :=\ \Sigma_{\widetilde M_g}\cap \left(\widetilde C_b\times \widetilde C_b\right).$$   
Denote by 
\begin{equation*}
\mathcal{X}\ \longrightarrow\ \mathbb{H}_{g-1}
\end{equation*}
the universal family of abelian varieties over the Siegel space, whose fibre over any $\tau'$ is 
$$X_{\tau'}\ \,=\ \,\frac{\mathbb{C}^{g-1}}{\mathbb{Z}^{g-1}+\tau' \mathbb{Z}^{g-1}}.$$
As in \eqref{epd}, the Prym--difference maps $\phi_b\,: \, \widetilde S_b\, \longrightarrow\, P(\pi_b:\widetilde C_b\to C_b)$ give a morphism
$$\phi\ :\ \widetilde{\mathcal{S}}\ \longrightarrow\
\mathcal{X},$$ making the  following diagram  commutative:
$$\begin{tikzcd}
\widetilde{\mathcal{S}} \arrow[r, "\phi"] \arrow[d, "\Phi"'] & \mathcal{X} \arrow[d] \\
\widetilde{M}_g \arrow[r, "\tau"] & \mathbb{H}_{g-1} .
\end{tikzcd}$$
The maps $\Phi$ and  $\tau$ are defined  in \eqref{ePh}  and \eqref{eq:prymperiod}  
respectively. There is a holomorphic line bundle 
$2\Theta\, \longrightarrow\, \mathcal{X}$ such that
$$(2\Theta)\big\vert_{X_{\tau'}}\ =\ 2\Theta_{\tau'}.$$
This line bundle $2\Theta$
is unique up to tensoring with a holomorphic line bundle pulled back from
$\mathbb{H}_{g-1}$. Since by Theorem \ref{prop:fundamentalpullback},
\begin{multline*}
(\phi^\ast 2\Theta)\big\vert_{\widetilde{S}_b}\ =\ \phi_b^\ast 2\Theta_{\tau(b)}\ \cong\ K_{\widetilde S_b}(2\widetilde\Delta_b-2\Sigma_b)\ =\\
=
\widetilde{\mathbf{p}}_1^\ast K_{\widetilde{\mathcal{C}}/\widetilde M_g}\otimes \widetilde{\mathbf{p}}_2^\ast K_{\widetilde{\mathcal{C}}/\widetilde M_g}\otimes \mathcal{O}_{\widetilde{\mathcal{S}}}(2\widetilde\Delta_{\widetilde M_g}-2\Sigma_{\widetilde{M}_g})
\ \, \Big\vert_{\widetilde S_b},
\end{multline*}
where $\widetilde{\mathbf{p}}_1$ and $\widetilde{\mathbf{p}}_1$ are maps in \eqref{wtpk}, 
and $\tau$ is the map in \eqref{eq:prymperiod}, we have
\begin{equation}
\label{eq:globaliso}
\phi^\ast 2\Theta\ \cong\ 
\widetilde{\mathbf{p}}_1^\ast K_{\widetilde{\mathcal{C}}/\widetilde M_g}\otimes \widetilde{\mathbf{p}}_2^\ast K_{\widetilde{\mathcal{C}}/\widetilde M_g}\otimes \mathcal{O}_{\widetilde{\mathcal{S}}}(2\widetilde\Delta_{\widetilde M_g}-2\Sigma_{\widetilde{M}_g})
\otimes \widetilde f^\ast L
\end{equation}
for some line bundle $L$ on $\widetilde M_g$ \cite[Section 10]{Mu2}.

\section {The \texorpdfstring{$\debar$}{}-derivative of the Prym projective structure}

The goal of this Section is the computation of $\debar \beta^P$. This will be accomplished on $\widetilde{M}_g$. The most difficult part is an expression for the restriction to $3\Deltas$ of the pullback of a section in $2\Xi$, see Lemma \ref{lemma:pull} below.
In \ref{subs:theta:debar} this computation is used --- together with several facts about theta functions --- to show that $\debar \beta^P$ is the pullback of the Fubini--Study metric through a Thetanullwert map and a period map, see Theorem \ref{teo:debar}.

\subsection{Restriction of the pullback}

In this section fix an \'etale double cover $\pi\,:\,\widetilde C\,\longrightarrow\, C$ and symplectic bases $\mathcal{B}$ and $\widetilde{\mathcal{B}}$, of $H_1(C,\, {\mathbb Z})$ and 
$H_1(\widetilde{C},\, {\mathbb Z})$ respectively, compatible with the covering as in \eqref{dabi}.
We refer to Section \ref{subs:topologicalconstruction} for all the notation in this Section.
We identify $(P,\,\Xi)$ with a quotient of $\mathbb{C}^{g-1}$ as in \eqref{eq:prymtorus}. Let
\begin{equation}\label{ecp1}
t_1,\, \cdots, \, t_{g-1}
\end{equation}
be the coordinates on $\mathbb{C}^{g-1}$. 
These coordinates on $\mathbb{C}^{g-1}$ and on $P$ will be used without further clarifications.
Also, we will identify $H^0(P,\,2\Xi)$ with  $V(\tau)$ (see \eqref{eq:secondorder}).

Our goal is to compute explicitly the following composition of homomorphisms:
\begin{equation*}
\begin{tikzcd}
H^0(P,\,2\Xi) \arrow{r}{\phi^*}   
&H^0(\widetilde C\times\widetilde C,\, \phi^* (2\Xi) \arrow{d}{\cong}  & \\
&H^0(\widetilde{S}, \tbb(-2\Sigma)) =H^0(\widetilde S,\, K_{\widetilde S}(2\widetilde \Delta-2\Sigma))\arrow {r}  &H^0(3\widetilde \Delta,\, K_{\widetilde S}(2\widetilde \Delta)\big\vert_{3\widetilde \Delta}).
\end{tikzcd}
\end{equation*}
Here $\tbb$ is defined in \eqref{deftbb} and the vertical isomorphism is given by Theorem  \ref {prop:fundamentalpullback}.

Denote by $$V_0(\tau) \ \subset\ V(\tau)$$ the subspace of sections that vanish at the origin. 

\begin{lemma}
There is a natural commutative diagram
\begin{equation}
\label{eq:dd}
\begin{tikzcd}
V_0(\tau) \arrow[d,hookrightarrow] \arrow[r, "\phi^\ast"] & H^0(\widetilde{S},\, K_{\widetilde S}(-2\Sigma)) \arrow[d,"f"] \arrow[r, "\big\vert_{\widetilde \Delta}"] & H^0(\widetilde{\Delta},\,
K_{\widetilde S}\big\vert_{\widetilde \Delta}) \arrow[d, "\iota"] \\
V(\tau) \arrow[r, "\phi^\ast"] & H^0(\widetilde{S},\, K_{\widetilde S}(2\widetilde \Delta-2\Sigma)) \arrow[r, "\big\vert_{3\widetilde \Delta}"] & H^0(3\widetilde \Delta, \, K_{\widetilde S}(2\widetilde \Delta)\big\vert_{3\widetilde \Delta}) 
\end{tikzcd} 
\end{equation}
The map $\iota$ will be specified in the proof.
\end{lemma}
\begin{proof}
By Proposition \ref{prop:fromizadi}  
\begin{equation*}
\phi^\ast
(V_0(\tau) ) \subset  H^0(\widetilde{S},\, K_{\widetilde S}(-2\Sigma)).
\end{equation*}
So the diagram makes sense. Moreover the left rectangle is obviously commutative. To prove the commutativity of the right part is harder.
 Observe that the   sequence:
\begin{gather*}
  0 \lra \OO_\SS / \OO_\SS(-\Deltas)
  \lra \OO_\SS (2\Deltas) / \OO_\SS(-\Deltas)
  \lra \OO_\SS (2 \Deltas) / \OO_\SS
  \lra 0
\end{gather*}
is exact by the third isomorphism theorem.
Since $\OO_\SS = \OO_\SS (2\Deltas) ( -2\Deltas)$ 
it can be rewritten as
\begin{gather*}
  0 \lra
  \OO_{\Deltas} \lra \OO_\SS (2\Deltas) \big\vert_{3\widetilde \Delta} \lra
  \OO_\SS (2\Deltas) \big\vert_{2\widetilde \Delta} \lra 0.
\end{gather*}
Tensoring with $K_\SS$ we obtain that the sequence of sheaves on $\SS$
\begin{gather}
\label  {eq:important}
  0 \lra
  K_\SS \big \vert _{\Deltas} \stackrel{\iota}{\lra} K_\SS (2\Deltas) \big\vert_{3\widetilde \Delta} \stackrel{q}{\lra}
  K_\SS (2\Deltas) \big\vert_{2\widetilde \Delta} \lra 0
\end{gather}
is exact. The map $\iota$ in \eqref {eq:dd} is the same as here but
on global sections.
Consider the following diagram of sheaves on $\SS$:
\begin{equation*}
  \begin{tikzcd}
    0 \arrow[r] & K_{\widetilde S}(-\widetilde\Delta-2\Sigma) \arrow[r, "\alpha"] \arrow[d,hookrightarrow, "i"] & K_{\widetilde S}(2\widetilde{\Delta}-2\Sigma) \arrow[r, "r_3"] \arrow[d,equal] & K_{\widetilde S}(2\widetilde{\Delta})\big \vert_{3\Deltas} \arrow[r] \arrow[d, two heads, "q"] & 0 \\
    0 \arrow[r] & K_{\widetilde S}(-2\Sigma) \arrow[r,"\beta"] & K_{\widetilde
      S}(2\widetilde{\Delta}-2\Sigma) \arrow[r] & K_{\widetilde
      S}(2\widetilde{\Delta})
\big \vert _{2\Deltas}
 \arrow[r]
    & 0
\end{tikzcd}
\end{equation*}
The  rows are exact since  $\Sigma$ and $\widetilde \Delta$ have no intersection.
Applying the snake Lemma to this diagram we get an isomorphism $\operatorname{coker} i \,\cong \, \ker q$.
 Observe that the composition $\big \vert _{3\Deltas} \circ f$ in
\eqref{eq:dd} is the map induced on global sections by
$r_3 \circ \beta$. Similarly the composition
$\iota \circ \big \vert_{\Deltas}$ is the map on global section
induced by the composition
\begin{gather*}
  K_\SS (-2\Sigma) \lra \operatorname{coker} i \cong \ker q \lra  K_{\widetilde S}(2\widetilde{\Delta})\big \vert_{3\Deltas}.
  \end{gather*}
  It follows from the diagram of sheaves above that this  map is equal to $r_3 \circ \beta$.   
This proves that the right rectangle of  \eqref{eq:dd} commutes.
\end{proof}

The following Lemma is at the basis of the computations in the rest paper.

\begin{lemma}\label{lemma:pull}
  Fix $\zeta\,\in\, V(\tau)\,=\, H^0(P,\,2\Xi)$ with
  $\zeta(0)\,\neq\, 0$, and an isomorphism
  $$E\ : \ \phi^\ast (2\Xi)\ \xrightarrow{\,\,\,\cong\,\,\,}\
  K_{\widetilde S}(2\widetilde{\Delta}- 2\Sigma).$$ Let
  $\alpha\,\neq\, 0$ be the biresidue of $ E(\phi^*\zeta)$ on the
  diagonal, that is,
  $E(\phi^\ast\zeta)\big\vert_{\widetilde \Delta}\, =\, \alpha \in
  \C$.  Then for all $\mu\, \in\, V(\tau)$
  \begin{equation}
    \label{eq:restrictiontodelta}
    E (\phi^\ast \mu)\big\vert_{3\widetilde\Delta}\ =\ \frac{\mu(0)}{\zeta(0)}\cdot E (\phi^\ast\zeta)\big\vert_{3\widetilde\Delta}+\frac{\alpha}{2\zeta(0)}\sum_{i,j=1}^{g-1} \left(\partial_i\partial_j\mu(0)-\frac{\mu(0)}{\zeta(0)}\partial_i\partial_j\zeta(0)\right)\gamma_i\gamma_j.
  \end{equation}
\end{lemma}
Note that the isomorphism $E$ exists by Theorem
\ref{prop:fundamentalpullback}.  Here
$H^0(\widetilde\Delta,\, K_{\widetilde S}\big\vert_{\widetilde
  \Delta})\,\cong\, H^0(\widetilde{C},\,2K_{\widetilde C})$ is
identified with a subspace of
$H^0(3\widetilde{\Delta},\, K_{\widetilde S}(2\widetilde
\Delta)\big\vert_{3\widetilde\Delta})$ through the morphism on global
sections induced by the map $\iota$ in \eqref{eq:important}.  See
\eqref{nega} for $\gamma_i$. The partial derivatives
$\partial_i$
are taken with respect to the coordinates $t_i$ defined in \eqref{ecp1}.
\begin{proof}
  The main point is proving \eqref{eq:restrictiontodelta} for 
  $\mu \in V_0(\tau)$. Fix such a $\mu$. We have to prove that
\begin{equation}
\label{eq:restrictiontov0}
E (\phi^\ast \mu)\big\vert_{3\widetilde\Delta}\ =\ \frac{\alpha}{2\zeta(0)}\sum_{i,j=1}^{g-1}\partial_i\partial_j\mu(0)\gamma_i\gamma_j.
\end{equation}
Set
\begin{gather*}
  f:= \frac{\mu}{\zeta}.
\end{gather*}
Since $\mu$ and $\zeta$ belong to $V(\tau)$, $f$ is  a meromorphic function on $P$.
Since $\zeta(0)\neq 0$,  $f$ is holomorphic near the origin. Since $\mu(0) = \partial_i\mu(0)=0$, we have
\begin{gather}
f(0) = \partial_if(0) =0, \qquad  \partial_i\partial_jf(0) = \frac{\partial_i\partial_j\mu(0) }{\zeta(0)},
\qquad   f(t) = \frac{1}{2} \partial_i\partial_j\mu(0) t_i t_j + o(t^2).
\label{fij}
\end{gather}
Define the map
 \begin{equation}\label{edpsi}
   \psi\ :=\ E \circ \phi^*\ :\ V(\tau)\ \longrightarrow\
 H^0(\widetilde{S}, \,K_{\widetilde S}(2\widetilde \Delta-2\Sigma)).
 \end{equation}
Then
\begin{gather}
  \label{phistarf}
  \frac{\psi(\mu)}{\psi(\zeta)} = \frac{ \phi^* (\mu)}{\phi^*(\zeta)} = \phi^* f.
\end{gather}
Fix a point $p\,\in \,C$ together with a holomorphic coordinate $z\,:\,U 
\,\longrightarrow\, \mathbb{C}$ such that $p\,\in\, U$ and $z(p)\,=\, 0$. This produces the coordinates $z_1 \,:=\, p_1^\ast z$ and $z_2 \,:=\, p_2^\ast z$ on $U\times U$.
We consider instead the coordinates $x\,=\,z_1+ z_2$ and $y\,=\,z_1-z_2$, which highlight the diagonal. 
Let $\mathcal R$ denote the locally defined curve in $\widetilde{S}$ given by the equation $x\,=\, 0$; so $y$ is a coordinate function on $\mathcal R$, and ${\mathcal R}\cap \widetilde\Delta\,=\, (p,\,p)$. The inclusion map $e\,:\,{\mathcal R}\,\hookrightarrow\, U\times U\,\hookrightarrow\, C\times C$ is given in coordinates by $$e(y)= \ (y/2,\,-y/2).$$
Let $\gamma_i\,=\,h_i dz$ on $U$. Identify $P$ with $\mathbb{C}^{g-1}/(\mathbb{Z}^{g-1}+\tau \mathbb{Z}^{g-1})$ as  in \eqref{eq:prymtorus}. Using \eqref{Prym-diff} 
and  the mean value theorem
we get
\begin{gather*}
  \phi\circ e (y)=\left(\int_{-y/2}^{y/2} \gamma_i\right)_{i=1,\cdots ,g-1} =  h(0)y+o(y),
\end{gather*}
where $h$ is the vector function with components $h_i$.
Hence
\begin{gather}
  \label{phistarf2}
  (\phi^*f ) (e(y)) = 
  f \circ \phi \circ e (y) =
  f (h(0)y + o(y)) = \frac{1}{2} \partial_i\partial_jf(0) \, h_i(0)h_j(0) \, y^2 + o(y^2).  
\end{gather}

Now,  
\begin{gather*}
  \sigma := \frac{dz_1\wedge dz_2}{(z_1-z_2)^2}
\end{gather*}
is a holomorphic
frame for the bundle $\bb(-2\Sigma)$ on $U \times U$.
Locally, $\psi(\zeta)$ may be written as
$$\psi(\zeta) (z_1,z_2) \ =\ (\alpha+(z_1-z_2)^2g(z_1,z_2)) \cdot \sigma$$
for some $g \,\in\, \mathcal O(U\times U)$.
Notice that there is no term in $(z_1-z_2)$ for residue reasons.
By (1) of \ref{prop:fromizadi}, $\psi(\mu)$ may be written as
$$ \psi(\mu)  (z_1,z_2)=\ (z_1-z_2)^2 q(z_1,z_2)\cdot \sigma$$
for some $q \,\in\, \mathcal O(U\times U)$.
Using \eqref{phistarf} it follows that
\begin{gather*}
  (\phi^*f ) (e(y)) = 
  \frac{\psi(\mu)} {\psi(\zeta)} ( e(y) ) = \frac
  { y^2 \ q(y/2, -y/2) } { \alpha + y^2 \ g(y/2,-y/2) }  = \frac{q(0,0)}{\alpha} \cdot y^2 + o(y^2).  
\end{gather*}
Comparing this with  \eqref{phistarf2}  
we get
\begin{gather*}
  \frac{q(0,0)}{\alpha} =
  \frac{1}{2} \partial_i\partial_j f(0) h_i(0)h_j(0) .
\end{gather*}
Using \eqref{fij} we obtain
\begin{gather*}
q(0,0)
   =
  \frac{\alpha}{2\zeta(0)} 
    \partial_i\partial_j\mu(0) h_i(0)h_j(0)  .
\end{gather*}
Multiplying by $dz_1\wedge dz_2$ 
this identity gives the following:
\begin{gather*}
\psi (\mu)|_{\widetilde\Delta}
\ =\   \left.\left(
  \frac{\alpha}{2\zeta(0)} \partial_i\partial_j\mu(0) 
  p_1^\ast \gamma_i\wedge p_2^\ast \gamma_j\right)\right|_{\widetilde \Delta}\ \ \, \text{at }(p,\,p)\,\in\,\widetilde\Delta ,  
\end{gather*}
where $|_{\widetilde\Delta}$ is the restriction in \eqref{eq:dd}. 
Since the point $p$ was chosen arbitrarily at the beginning of the computation,  this proves that
$$\left.\left(\frac{\alpha}{2 \zeta(0)}\sum_{i,j=1}^{g-1} \partial_i\partial_j \mu(0)p_1^\ast \gamma_i\wedge p_2^\ast \gamma_j\right)\right|_{\widetilde \Delta}\ =\ (E\phi^\ast \mu)|_{\widetilde\Delta}$$
for all $\mu\,\in \,V_0(\tau)$.

Comparing with \eqref{eq:dd} and using the isomorphism $H^0(\widetilde{\Delta},\, K_{\widetilde S}\big\vert_{\widetilde \Delta})\,\cong\, H^0(\widetilde{C}, \,2K_{\widetilde C})$, the proof of \eqref{eq:restrictiontov0}
is completed.

Take now any $\mu\,\in\, V(\tau)$. We have
$$\left .  E \phi^\ast \mu \right|_{3\widetilde \Delta}\ =\ \left . E \phi^\ast \left (\frac{\mu(0)}{\zeta(0)}\zeta\right)\right|_{3\widetilde \Delta}+\left .E\phi^\ast\left(\mu-\frac{\mu(0)}{\zeta(0)}\zeta\right)\right|_{3\widetilde \Delta}$$
and $\mu-\frac{\mu(0)}{\zeta(0)}\zeta\,\in\, V_0(\tau)$, so \eqref{eq:restrictiontodelta} follows from \eqref{eq:restrictiontov0}.
\end{proof}

\subsection{Theta functions and computation of \texorpdfstring{$\debar$}{}-derivative}
\label{subs:theta:debar}

We now proceed to compute $\debar\beta^P$, the derivative of the Prym projective structure constructed in \eqref{eq:defbetap}. First, some notation will be set up. Denote by $\mathbb{Z}_2$ the cyclic group $\mathbb{Z}/2\mathbb{Z}$ of order $2$.
For $u\,\in\, \mathbb{Z}_2^g$ and $\tau\,\in\,\mathbb{H}_g$, define
\begin{equation}
\label{eq:thetau}
\theta_u(\tau;\,z)\ := \ \sum_{m\in\mathbb{Z}^g} \exp({2\pi \sqrt{-1} (m+u/2)^t\tau (m+u/2)+4\pi \sqrt{-1} (m+u/2)^tz}).
\end{equation}
where, on the right--hand side, we substitute $u$ with a representative in $\mathbb{Z}^g$. It is easy to see that the right--hand side of \eqref{eq:thetau} does not depend on the representative of $u$.
Then $$\{\theta_u(\tau;z)\,\, \big\vert\,\, u\,\in\, \mathbb{Z}_2^g\}$$
is a basis of $V(\tau)$, see e.g. \cite[\S 4.1.2]{VG}. In addition to that, it is orthogonal of constant length with respect to the Hermitian structure on $V(\tau)$; see \cite[p.~80]{IG}.
Define
\begin{equation}\label{Thg}
\Theta_g\ :\ \mathbb{H}_g\ \lra\ \mathbb{P}^{2^g-1}, \ \ \, \Theta_g(\tau):=\  [\theta_u(\tau;\,0)]_{u\in \mathbb{Z}_2^g}.
\end{equation}
Given a projective space
 $\mathbb{P}^N$ with a fixed basis we consider the Fubini--Study metric:
\begin{equation}
\label{defFS}
\omega_{FS}\ =\ \frac{\sqrt{-1}}{2}\partial \debar \log |X|^2,
\end{equation}
where $|X|^2\,=\,\sum_{i=0}^N |X_i|^2$ and $X_0,\, \cdots ,\, X_N$ are the homogeneous coordinates.

\begin{theorem}\label{teo:debar}
Identify $f_\ast K_{\mathcal{C}/B}^{\otimes 2}$ with $\Omega^1_{\widetilde M_g}$ through the isomorphism $\lambda$ in \eqref{eq:lambdaiso}. 
Let $\tau$ be the map defined in \eqref{eq:prymperiod},
and let $\beta^P$ be  the section constructed in \eqref{eq:defbetap}.
Then over $\widetilde M_g$ we have the following identity:
$$\debar \beta^P\ =\ 8 \pi \tau^\ast \Theta_{g-1}^\ast \omega_{FS}.$$
\end{theorem}
\begin{proof}
Fix a point $b_0\,\in\,\widetilde M_g$ and a  neighborhood $B$ of $b_0$ in $
\widetilde{M}_g$ which is biholomorphic to a polydisk. Then the line bundle $L$ in \eqref{eq:globaliso} is trivial when restricted to $B$. Hence we can fix an isomorphism
\begin{gather*} E\ :\ \phi^\ast 2 \Theta\big\vert_{\widetilde f^{-1}(B)}\ \xrightarrow{\,\,\,\cong\,\,\,} \ \left.\widetilde{\mathbf{p}}_1^\ast K_{\widetilde{\mathcal{C}}/\widetilde M_g}\otimes \widetilde{\mathbf{p}}_2^\ast K_{\widetilde{\mathcal{C}}/\widetilde M_g}\otimes \mathcal{O}_{\widetilde{\mathcal{S}}}(2\widetilde\Delta_{\widetilde M_g}-2\Sigma_{\widetilde{M}_g})\,\right|_{\widetilde f^{-1}(B)}.
\end{gather*}
Choose an element $u_0\,\in\, \mathbb{Z}_2^{g-1}$ such that $\theta_{u_0}(\tau(b_0);\,0)
\,\neq \,0$. 
This is possible because $\{\theta_u\}_{u\in\mathbb{Z}_2^{g-1}}$ is a basis of $V(\tau)\,\cong\, H^0(P,\,2\Xi)$ and $|2\Xi|$ is base point-free.
Shrinking $B$ --- if necessary --- we can assume that $\theta_{u_0}(b;\,0)\,\neq\, 0$ for all $b\,\in \,B$. 
The restriction to the fibre over $b$ yields an isomorphism
$$E_b\ :\ \phi_b^\ast 2 \Theta_b\ \longrightarrow\
K_{\widetilde S_b}(2\widetilde \Delta_b-2\Sigma_b).$$
Define
$$\alpha(b) \ :=\ (E_b \phi_b^\ast \theta_{u_0}(\tau(b),\,\cdot))\big\vert_{\widetilde \Delta_b}.$$
Let $s_b$ be any element belonging to the line $\mathscr L \subset
\,\in\, H^0(X_{\tau(b)},\, 2\Theta_{\tau(b)})$   defined in \eqref{eq:defines} and depending smoothly on $b$.
Then we get the following expression for the section $\beta^P$ of $\mathcal V$ that was defined in \eqref{eq:defbetap}: 
$$\beta^P(b)\ =\ \frac{1}{(E_b\phi_b^\ast s_b)\big\vert_{\widetilde{\Delta}_b}} (E_b\phi_b^\ast s_b)\big\vert_{3\widetilde{\Delta}_b}.$$
Applying Lemma \ref{lemma:pull} with $\zeta = \theta_{u_0}(\tau(b);0)$ we obtain
\begin{align*}
(E_b \phi_b^\ast s_b)\big\vert_{3\widetilde\Delta_b}\ &=\ \frac{s_b(0)}{\theta_{u_0}(\tau(b);\,0)}(E_b \phi_b^\ast \theta_{u_0}(\tau(b),\,\cdot))\big\vert_{3\widetilde\Delta_b}+\\
&+\frac{\alpha(b)}{2 \theta_{u_0}(\tau(b);0)}\sum_{i,j=1}^{g-1} \left(\partial_i\partial_j s_b(0)-\frac{s_b(0)}{\theta_{u_0}(\tau(b);\,0)}\partial_i\partial_j\theta_{u_0}(\tau(b);\,0)\right)\gamma_i(b)\gamma_j(b),\\
(E_b\phi_b^\ast s_b)\big\vert_{\widetilde \Delta_b} \ &=\ \frac{s_b(0)}{\theta_{u_0}(\tau(b);\,0)}(E_b \phi_b^\ast \theta_{u_0}(\tau(b),\,\cdot))\big\vert_{\widetilde\Delta_b} \ =\ \frac{\alpha(b) s_b(0)}{\theta_{u_0}(\tau(b);0)},\\
\beta^P(b) \ &=\ \frac{1}{(E_b\phi_b^\ast s_b)\big\vert_{\widetilde{\Delta}_b}} (E_b\phi_b^\ast s_b)\big\vert_{3\widetilde{\Delta}_b} =\\
&=\frac{1}{\alpha(b)} (E_b \phi_b^\ast \theta_{u_0}(\tau(b);\,\cdot))\big\vert_{3\widetilde\Delta_b}+\\
&+\frac{1}{2}\sum_{i,j=1}^{g-1} \left(\frac{\partial_i\partial_j s_b(0)}{s_b(0)} -\frac{1}{\theta_{u_0}(\tau(b);0)} \partial_i\partial_j\theta_{u_0}(\tau(b);\,0)\right)\gamma_i(b)\gamma_j(b). 
\end{align*}
Since $\theta_{u_0}$ and the $\gamma_i$'s vary holomorphically with $b$, the terms
\begin{gather*}
    \frac{1}{\alpha(b)} (E_b \phi_b^\ast \theta_{u_0}(\tau(b);\,\cdot ))\big\vert_{3\Delta}\qquad \text{and } \qquad \frac{1}{\theta_{u_0}(\tau(b);\,0)} \partial_i\partial_j\theta_{u_0}(\tau(b);\,0)
\end{gather*}
represent holomorphic sections of the  vector bundle $\mathcal V $ defined in \eqref{defV}.
Consequently, if we identify $\tau^\ast d Z_{ij}\, =\,\lambda(-\gamma_i\gamma_j)$ with $-\gamma_i\gamma_j$ (see \eqref{eq:lambdaiso} and Proposition \ref{prop:codifferential}) we get 
\begin{equation}
\label{eq:protodebar}
\debar \beta^P(b)\ =\ -\debar\biggl ( \frac{1}{2\, s_b(0)} \sum_{i,j=1}^{g-1}
\partial_i\partial_j s_b(0)\tau^\ast d Z_{ij} \biggr ).
\end{equation}
Since $\{\theta_u(\tau(b),\,\cdot)\}_{u\in U}$ is orthogonal of constant length, a possible choice for $s_b$ is the following:
\begin{gather} 
\label{defsb}
s_{b}(t)\ =\ \sum_{u\in\mathbb{Z}_2^{g-1}} \overline{\theta_u(\tau(b);\,0)}
\theta_u(\tau(b);\,t).
\end{gather}
In fact this section is smooth in $b$ and is orthogonal to the sections vanishing at $0$, by the same simple argument used in \cite[Lemma 6.4]{BGV}.
We will complete the proof by substituting \eqref{defsb} into \eqref{eq:protodebar}. The last fact we need is the heat equation for second order theta functions:
$$4\pi\sqrt{-1} (1+\delta_{ij}) \frac{\partial \theta_u(Z;\,t)}{\partial Z_{ij}}
\ =\ \frac{\partial^2 \theta_u(Z;\, t)}{\partial t_i\partial t_j},$$
where $\{Z_{ij}\}_{1\leq i\leq j\leq g-1}$ are the coordinates on $\mathbb{H}_{g-1}$ and $t_i$ are the coordinates on $\mathbb{C}^{g-1}$; see \cite[p.~20]{VG}. For simplicity of notation we will drop $\tau(b)$ when it is not important. We get
\begin{align*}
s_b(0) &= 
\sum_u|\theta_u(0)|^2, \\
\partial_i\partial_j s_b(0) & = 
\sum_u \overline{\theta_u(0)}\, \partial_i \partial_j\theta_u(0),\\
\sum_{i,j=1}^{g-1} 
\partial_i\partial_j s_b(0)\tau^\ast d Z_{ij}\ &=\ 
\sum_{i,j=1}^{g-1} 
\sum_u\overline{\theta_u(0)}  \frac{\partial^2  \theta_u(0)}{\partial t_i \partial t_j} \tau^\ast d Z_{ij}\\
&=\ 8\pi \sqrt{-1}\sum_u
\sum_{i=1}^{g-1}
\overline{\theta_u(0)}\frac{\partial \theta_u(0)}{\partial Z_{ii}}\tau^\ast d Z_{ii}\, + \\
&+8\pi \sqrt{-1}\sum_u\sum_{1\leq i<j\leq g-1} \overline{\theta_u(0)} \frac{\partial \theta_u(0)}{\partial Z_{ij}}\tau^\ast d Z_{ij}
\\
&=\ 8\pi \sqrt{-1}\sum_u\sum_{1\leq i\leq j\leq g-1}  \frac{\partial}{\partial Z_{ij}} \big\vert\theta_u(\tau(b);0)\big\vert^2 \tau^\ast d Z_{ij}
\end{align*}
Now $|\theta_u(\cdot;0)|^2$ is a function on $\mathbb H_{g-1}$,  so
\begin{align*}
    \sum_{1\leq i\leq j\leq g-1}  \frac{\partial}{\partial Z_{ij}} \big\vert\theta_u(\cdot ;0)\big\vert^2 
    d Z_{ij} & = \partial \, |\theta_u(\cdot;0)|^2,\\
    \sum_{1\leq i\leq j\leq g-1}  \frac{\partial}{\partial Z_{ij}} \big\vert\theta_u(\tau(b);0)\big\vert^2 \tau^\ast d Z_{ij} &= \tau^* \partial \, |\theta_u( \cdot;0)|^2  \\
\sum_{i,j=1}^{g-1} 
\partial_i\partial_j s_b(0)\tau^\ast d Z_{ij}\ &=\ 
    8 \pi \sqrt{-1}  \  \tau^* \partial \left ( \sum_u|\theta_u( \cdot;0)|^2 \right ) 
\\
\frac{1}{2\, s_b(0)} \sum_{i,j=1}^{g-1}
\partial_i\partial_j s_b(0)\tau^\ast d Z_{ij} &=
\frac{4\pi \sqrt{-1}}{ \sum_u |\theta_u(\tau(b);0)|^2 }
\tau^* \partial  \left ( \sum_u|\theta_u( \cdot;0)|^2 \right ) \\& =
4 \pi \sqrt{-1} 
\tau^* \partial  \log \left ( \sum_u|\theta_u( \cdot;0)|^2 \right ) .
\end{align*}

Plugging this in \eqref{eq:protodebar} and recalling \eqref{defFS} yields 

\begin{align*}
\debar \beta^P\ &=\ 4\pi \sqrt{-1}\, \partial  \debar  \log 
\left ( \sum_u|\theta_u( \cdot;0)|^2 \right ) =\\
&=
\ 4\pi \sqrt{-1}\, \partial  \debar \,\Theta_{g-1}^\ast  
\left ( \sum_u |X_u|^2 \right)= 
\\
&=
\ 8\pi \, \tau^\ast \Theta_{g-1}^\ast \omega_{FS}.    
\end{align*}
This completes the proof.
\end{proof}

\begin{remark}
\label{rem:debarwithpg}
The map $\Theta_{g-1}:\mathbb{H}_{g-1}\longrightarrow \mathbb{P}^{2^{g-1}-1}$ in \eqref{Thg} does not factor through the quotient $\mathbb{H}_{g-1}\longrightarrow A_{g-1}$.
It does factor through the finite cover $A_{g-1}(2,4)$ of $A_{g-1}$, see \cite[p.28]{VG}.
However, the form $\Theta_{g-1}^\ast \omega_{FS}$ descends to $A_{g-1}$, see \cite[Section 3]{BV}.
So  using the diagram in \eqref{RPg} the result of Theorem \ref{teo:debar} may be rewritten as the following identity on $R_g$:
\begin{equation}
 \debar \beta^P \ =\ 8\pi P_g^\ast \Theta_{g-1}^\ast \omega_{FS}
\end{equation}
\end{remark}

\section{The Schottky-Jung identity and Question \ref{question:one} }

In this Section we start by fixing some notation  for theta functions and Thetanullwert maps. 
Next we recall the classical Schottky-Jung identity and we use it  to obtain a new expression for $\debar \beta^P$.
This new expression gives $\debar\beta^P$ in terms of the period matrix of the curve $C$, i.e., the base of the covering; see Corollary \ref{cor:equivalentdebar}. 
Relying on this expression, we are able to apply a degeneration argument that shows that $\debar \beta^P$  is not the pullback of a $(1,1)$-form
on $M_g$. 
This implies that also $\beta^P$ is truly an object on $R_g$, answering Question \ref{question:one}.

\subsection{The Schottky-Jung identity and the \texorpdfstring{$\debar$}{}-derivative}

For $\varepsilon,\delta\in \mathbb Z_2^g$ set
\begin{equation*}
\varepsilon\cdot\delta\, :\,=\,\sum_{i=1}^g \varepsilon_i\delta_i.
\end{equation*}
Let $\tau\,\in\, \mathbb{H}_{g-1}$ and $\Pi\,\in\, \mathbb{H}_g$. For every $(\varepsilon,\, \varepsilon')\,\in\, \mathbb{Z}_2^{g-1}\times\mathbb{Z}_2^{g-1}$ define the theta function
\begin{equation}
\label{eq:thetachar}
\theta\begin{bmatrix}
 \varepsilon\\
 \varepsilon'
\end{bmatrix}(\tau;z)\ :=\ \sum_{m\in \mathbb{Z}^{g-1}} \exp({\pi\sqrt{-1} (n+\varepsilon/2)^t\, \tau \, (n+\varepsilon/2)+2\pi\sqrt{-1} (n+\varepsilon/2)^t(z+\varepsilon'/2)}).
\end{equation}
This is the same notation as in \cite{GSM}, while it is different from the one in \cite{BV}, where 
this function is denoted by $\theta\begin{bmatrix}
 \varepsilon/2\\
 \varepsilon'/2
\end{bmatrix}.$
This is not exactly a rigorous definition because the member on the right--hand side depends on the representatives of $\varepsilon$ and $\varepsilon'$ in $\mathbb{Z}^{g-1}$ up to a sign. This will not be a problem for our purposes. It is easy to see from \eqref{eq:thetachar} that 
\begin{equation}
\label{eq:thetaparity}
\theta\begin{bmatrix}
 \varepsilon\\
 \varepsilon'
\end{bmatrix}(\tau;\,-z)\ =\ (-1)^{\varepsilon\cdot \varepsilon'}\theta\begin{bmatrix}
 \varepsilon\\
 \varepsilon'
\end{bmatrix}(\tau;\,z).
\end{equation}

We will also need the following theta functions:

\begin{equation}
\theta\begin{bmatrix}
 0 & \varepsilon \\
 0 & \varepsilon'\
\end{bmatrix}(\Pi;\,z)\ :=\ \theta\begin{bmatrix}
 (0,\varepsilon_1,\cdots ,\varepsilon_{g-1})\\
 (0,\varepsilon_1',\cdots ,\varepsilon_{g-1}')
\end{bmatrix} (\Pi;\,z)
\end{equation}

\begin{equation}
\theta\begin{bmatrix}
 0 & \varepsilon \\
 1 & \varepsilon'\
\end{bmatrix}(\Pi;\,z)\ :=\ \theta\begin{bmatrix}
 (0,\varepsilon_1,\cdots ,\varepsilon_{g-1})\\
 (1,\varepsilon_1',\cdots ,\varepsilon_{g-1}')
\end{bmatrix} (\Pi;\,z).
\end{equation}

Again, these functions are well defined up to a sign. However, it is easy to prove by inspecting \eqref{eq:thetachar} that the following products
$$\theta^2\begin{bmatrix}
 \varepsilon\\
 \varepsilon'
\end{bmatrix}(\tau;\,z),\ \ \, 
\theta\begin{bmatrix}
 0 & \varepsilon \\
 0 & \varepsilon'\
\end{bmatrix}(\Pi;\,z)\theta\begin{bmatrix}
 0 & \varepsilon \\
 1 & \varepsilon'\
\end{bmatrix}(\Pi;\,z)$$
are well defined and are actually independent of the representatives of $\varepsilon,\, \varepsilon'$ chosen in the right--hand side of \eqref{eq:thetachar}.

Using these theta functions one can define Thetanullwert maps to the same projective space. Namely denote by $E_{g-1}$ the set of even $g-1$ characteristics, i.e.,
$$E_{g-1}\ :=\ \left\{(\varepsilon,\, \varepsilon')\,\in\, \mathbb{Z}_2^{g-1}\times
\mathbb{Z}_2^{g-1}\,\, \big\vert\,\, \sum \varepsilon_i\varepsilon'_i\,\equiv\, 0\, \mod\, 2\right\}.$$
Let $N + 1 $ be the number of such characteristics, so we have  
$$
N + 1 \,=\, \frac{4^{g-1}+2^{g-1}}{2}.
$$
Then we define
\begin{align}
\label{b1}
&\theta_{g-1}^2\ :\ \mathbb{H}_{g-1}\ \longrightarrow\ \mathbb{P}^N, \ \ \, \qquad
\theta_{g-1}^2(\tau):=\ \left[\theta^2\begin{bmatrix}
 \varepsilon\\
 \varepsilon'
\end{bmatrix}(\tau;\,0)\right]_{(\varepsilon,\varepsilon')\in E_{g-1}}
\\
\label{b2}
&\Theta_{g}^{SJ}\ :\ \mathbb{H}_{g} \ \longrightarrow\
\mathbb{P}^N, \ \ \, \qquad 
\Theta_g^{SJ}(\Pi):= \left[\theta\begin{bmatrix}
 0 & \varepsilon \\
 0 & \varepsilon'\
\end{bmatrix}(\Pi;0)\theta\begin{bmatrix}
 0 & \varepsilon \\
 1 & \varepsilon'\
\end{bmatrix}(\Pi;0)\right]_{(\varepsilon,\varepsilon')\in E_{g-1}}.
\end{align}

It is known that these maps are well defined, that is, the thetanulls that describe them do not all vanish simultaneously. We will explicitly prove this in Corollary \ref{cor:welldef}. 

The Schottky--Jung identities link the maps $\theta^2_{g-1}$ and $\Theta^{SJ}_{g}$.
We will use this classical result as stated in \cite[Theorem 1]{FR}.

\begin{theorem}[{Schottky--Jung}]\label{TSJ}
Denote by $\Pi\,:\,\widetilde{M}_g\, \longrightarrow\, \mathbb{H}_g$ the period map (see \eqref{eq:periodmatrix}), and denote by $\tau\,:\,\widetilde{M}_g\,\longrightarrow\,
\mathbb{H}_{g-1}$ the Prym period map (see \eqref{eq:prymperiod}).
Then the following diagram is commutative:
$$\begin{tikzcd}
\widetilde M_g \arrow[r, "\Pi"] \arrow[d, "\tau"'] & \mathbb{H}_g \arrow[d, "\Theta_g^{SJ}"] \\
\mathbb{H}_{g-1} \arrow[r, "\theta^2_{g-1}"'] & \mathbb{P}^N 
\end{tikzcd}
$$
where $\theta_{g-1}^2$ and $\Theta_g^{SJ}$ are constructed in \eqref{b1} and \eqref{b2}
respectively.
\end{theorem}

The maps $\theta^2$ and $\Theta$ have been studied by several authors, see \cite{SM} and references therein; these maps are known to be well defined and projectively equivalent up to a Veronese embedding.

For $$u \,=\, (u_1,\,\cdots ,\, u_{g-1}) \,\in\, \mathbb{Z}_2^{g-1}$$ and
$\Pi\,\in\, \mathbb{H}_g$ set
\begin{gather}
\notag
\theta_u'(\Pi;\,z)\ :=\ \theta\begin{bmatrix}
 0 & u\\
 1 & 0
\end{bmatrix}(2\Pi;\, 2z)\ =\ \theta \begin{bmatrix}
 (0,u_1,\cdots ,u_{g-1})\\
 (1,0,\cdots ,0)
\end{bmatrix}(2\Pi;\, 2z),\\
\label{b3}
\Theta'_g\ :\ \mathbb{H}_g\ \longrightarrow\ \mathbb{P}^{2^{g-1}-1}, \ \ \, \Theta'_g(\Pi) := \ [\theta_u'(\Pi;\,0)]_{u\in \mathbb{Z}_2^{g-1}} \ =\ \left[\theta\begin{bmatrix}
 0 & u\\
 1 & 0
\end{bmatrix}(2\Pi;\, 0)\right]_{u\in \mathbb{Z}_2^{g-1}}.
\end{gather}

Lastly, denote by $v_2\,:\,\mathbb{P}(\mathbb{C}^n)\, \longrightarrow\,
\mathbb{P}(\Sym^2\mathbb{C}^n)$ the ``isometric" Veronese embedding induced by the linear map
\begin{equation}\label{b4}
v_2\left(\sum_{i=1}^n \lambda_i e_i\right) \ =\ \sum_{i=1}^n \lambda_i^2 e_i\odot e_i+\sum_{1\leq i<j\leq n} \sqrt{2}\lambda_i\lambda_j e_i\odot e_j.
\end{equation}
This map $v_2$ enjoys the following property: if $\omega_{FS}^N$ is the Fubini--Study metric on $\mathbb{P}(\Sym^2\mathbb{C}^n)$, and $\omega_{FS}^n$ is the Fubini--Study metric on $\mathbb{P}(\mathbb{C}^n)$, then
\begin{equation}\label{eq:isometricveronese}
v_2^\ast\, \omega_{FS}^N\ =\ 2 \ \omega_{FS}^n.
\end{equation}
See \cite[Lemma 3.1]{BV} for a proof.

\begin{proposition}\label{prop:welldef}
There is a holomorphic and isometric automorphism $p\,:\, \mathbb{P}^N\,\longrightarrow\,
\mathbb{P}^N$ that makes the following diagram  commutative:
\begin{equation}
\label{eq:pyramid}
\begin{tikzcd}
\mathbb{H}_{g-1} \arrow[r, "\Theta_{g-1}"] \arrow[rd, "\theta_{g-1}^2"'] & \mathbb{P}^{2^{g-1}-1} \arrow[d, "p\circ v_2"] & \mathbb{H}_g \arrow[l, "\Theta_g'"'] \arrow[ld, "\Theta_g^{SJ}"] \\
 & \mathbb{P}^N .&
\end{tikzcd}
\end{equation}
\end{proposition}
\begin{proof}
The commutativity of the left triangle was already observed in  \cite[Lemma 3.2]{BV}, but it is convenient to prove in one shot that there is a vertical map  that makes both triangles commute.

The Riemann summation formula (\cite[p.~139]{IG}) will be used 
in an essential way. Using the notation of \eqref{eq:thetachar} this formula says that for all $\alpha,\,\beta,\,\varepsilon$ in $\mathbb{Z}_2^g$,\, $Z\,\in\,\mathbb{H}_g$,\, and $z,\,x\,\in\,\mathbb{C}^g$, 
\begin{multline}
\label{addition}
\theta\begin{bmatrix}
 \alpha\\
 \beta
\end{bmatrix}(2Z;\,2z)\cdot  
\theta\begin{bmatrix}
 \alpha+\varepsilon\\
 \beta
\end{bmatrix}(2Z;\,2x)\\ 
=\ 
\frac{1}{2^g} \sum_{\sigma\in\mathbb{Z}_2^g} (-1)^{\alpha\cdot \sigma} 
\ \theta\begin{bmatrix}
\varepsilon\\
\beta+\sigma
\end{bmatrix}(Z;\,z+x)
\cdot
\theta\begin{bmatrix}
\varepsilon\\
\sigma
\end{bmatrix}(Z;\,z-x).
\end{multline}
Notice that $\theta_u(\Pi;\,z)\,=\,\theta\begin{bmatrix}
u\\
0
\end{bmatrix}(2\Pi;\,2z)$. 
Applying the summation formula with $z=x=0$ we get
\begin{equation}\label{b8}
\theta_u(\Pi;\,0)\cdot \theta_{u+\varepsilon}(\tau;\,0)\ =\ \frac{1}{2^{g-1}}\sum_{\sigma\in\mathbb{Z}_2^{g-1}} (-1)^{u\cdot \sigma}\theta^2\begin{bmatrix}
\varepsilon\\
\sigma
\end{bmatrix}(0;\,\tau)
\end{equation}
for all $\Pi\,\in\,\mathbb{H}_{g-1}$ and for all $u\,\in\,\mathbb{Z}_2^{g-1}$.
Substituting $u'$ for $u+\varepsilon$ in \eqref{b8} 
and recalling from   \eqref{eq:thetaparity} that
$$
\theta^2\begin{bmatrix}
u+u'\\
\sigma
\end{bmatrix}(\Pi;\,0)\ =\ 0,
$$ 
when $(u+u')\cdot \sigma\,\equiv\,1\mod 2$,
we get 
\begin{multline*}
\theta_u(\Pi;\,0)\theta_{u'}(\Pi;\,0)\ =\ \frac{1}{2^{g-1}}\sum_{\sigma\in\mathbb{Z}_2^{g-1}} \theta^2\begin{bmatrix}
 u + u'\\
 \sigma
\end{bmatrix}(\Pi;\, 0)\ =\\
=\ \frac{1}{2^{g-1}}\sum_{\sigma\in\mathbb{Z}_2^g: (u+u',\sigma)\in E_{g-1}}
(-1)^{u\cdot \sigma} \theta^2\begin{bmatrix} u+u'\\
\sigma
\end{bmatrix}(\Pi;\,0),
\end{multline*}
Therefore,
 \begin{equation}
 \label{eq:emme1}
 \sqrt{2}\theta_{u}(0;\,\tau)\theta_{u'}(0;\,\tau)\,\,=\,\,
\sum_{(\varepsilon,\sigma)\in E_{g-1}} M_{u,u';\varepsilon,\sigma}\theta^2\begin{bmatrix}
 \varepsilon\\
 \sigma
\end{bmatrix}(\Pi;\,0)\ \ \, (u,\, u'\,\in\, \mathbb{Z}_2^{g-1},\
u\, \not=\, u')
 \end{equation}
 \begin{equation}
 \label{eq:emme2}
 \theta_{u}(\Pi;\,0)^2\,\,=\,\,\sum_{(\varepsilon,\sigma)\in E_{g-1}} M_{u,u;\varepsilon,\sigma}\theta^2\begin{bmatrix}
 \varepsilon\\
 \sigma
 \end{bmatrix}(\Pi;\,0)\
\ \ (u\,\in\, \mathbb{Z}_2^{g-1}),
 \end{equation}
 where
\begin{equation}
\label{eq:matrixM}
M_{u,u';\varepsilon,\sigma}\,\,\,=\,\,\,
 \begin{cases}
 \frac{1}{2^{g-1}}(-1)^{u\cdot \sigma} & \,\text{if }\,u+u'\,=\,\varepsilon\,=\,0\\
 \frac{\sqrt{2}}{2^{g-1}}(-1)^{u\cdot \sigma} & \,\text{if }\,u+u'\,=\,\varepsilon\,\neq\, 0\\
 0 &\, \text{otherwise.}
 \end{cases} 
\end{equation}
Here $M$ should be understood as a matrix with $(4^{g-1}+2^{g-1})/{2}$
columns  indexed by
even characteristics $(\varepsilon,\,\sigma)$ and $(4^{g-1}+2^{g-1})/{2}$ rows indexed by (unordered) pairs of elements $(u,\,u')$ of $\mathbb{Z}_2^{g-1}$. By \eqref{eq:emme1} and \eqref{eq:emme2}, the automorphism $p$ induced by $M$ satisfies $p\circ v_2\circ \Theta_{g-1}=\theta^2_{g-1}$.

The following notation will be used: we will write $\widetilde\sigma\,\in\,\mathbb{Z}_2^g$ as $(\sigma_1,\,\sigma)$ with $\sigma_1\,\in\,\mathbb{Z}_2$ and $\sigma\,\in\,\mathbb{Z}_2^{g-1}$. Applying the summation formula \eqref{addition} again, we get
\begin{align*}
\theta_u'(\Pi;\,0)\cdot &\theta'_{u+u'}(\Pi;\,0)\\
&=\ \frac{1}{2^g}\sum_{\widetilde\sigma\in\mathbb{Z}_2^g}(-1)^{(0,u)\cdot (\sigma_1,\sigma)}\theta\begin{bmatrix}
0 & u+u'\\
1+\sigma_1 & \sigma
\end{bmatrix}(\Pi;0)
\theta\begin{bmatrix}
0 & u+u'\\
\sigma_1 & \sigma
\end{bmatrix}(\Pi;\,0)\\
&=\ \frac{2}{2^g}\sum_{\sigma\in\mathbb{Z}_2^{g-1}} (-1)^{u\cdot \sigma} 
\theta\begin{bmatrix}
0 & u+u'\\
1 & \sigma
\end{bmatrix}(\Pi;\,0)
\theta\begin{bmatrix}
0 & u+u'\\
0 & \sigma
\end{bmatrix}(\Pi;\,0)\\
&=\ \frac{1}{2^{g-1}}\sum_{\sigma\in\mathbb{Z}_2^{g-1}: (u+u',\sigma)\in E_{g-1}} (-1)^{u\cdot \sigma} 
\theta\begin{bmatrix}
0 & u+u'\\
1 & \sigma
\end{bmatrix}(\Pi;\,0)
\theta\begin{bmatrix}
0 & u+u'\\
0 & \sigma
\end{bmatrix}(\Pi;\,0).
\end{align*}
Then we have
\begin{equation*}
 \sqrt{2}\, \theta'_{u}(0;\,\Pi)\cdot\theta'_{u'}(0;\,\Pi)\,\,=\,\,
\sum_{(\varepsilon,\sigma)\in E_{g-1}} M_{u,u';\varepsilon,\sigma}
\theta\begin{bmatrix}
0 & \varepsilon\\
1 & \sigma
\end{bmatrix}(\Pi;\,0)
\theta\begin{bmatrix}
0 & \varepsilon\\
0 & \sigma
\end{bmatrix}(\Pi;\,0)
\end{equation*}
\begin{equation*}
 \theta_{u}'(\Pi;\,0)^2\,\,=\,\,\sum_{(\varepsilon,\sigma)\in E_{g-1}} M_{u,u;\varepsilon,\sigma}
 \theta\begin{bmatrix}
0 & \varepsilon\\
1 & \sigma
\end{bmatrix}(\Pi;\,0)
\theta\begin{bmatrix}
0 & \varepsilon\\
0 & \sigma
\end{bmatrix}(\Pi;\,0)
\end{equation*}
for $u\,\neq\, u'\,\in\,\mathbb{Z}_2^{g-1}$; here $M$ is the same as in \eqref{eq:matrixM}. Then, exactly as before, we deduce that $p\circ v_2\circ \Theta_g'\,=\,\Theta_g^{SJ}$, and the diagram in \eqref{eq:pyramid} commutes.

To prove that $p$ is an isometry of the Fubini--Study metric, it is enough to show that $M$ is the multiple of a unitary matrix. This is in fact a straightforward computation; we refer to the proof of \cite[Lemma 3.2]{BV}.
\end{proof}

\begin{corollary}
\label{cor:welldef}
The maps $\Theta_{g-1},\, \Theta_g',\, \theta_{g-1}^2$\, and $\Theta_g^{SJ}$ are all well defined.
\end{corollary}

\begin{proof}
In view of \eqref{eq:pyramid}, it suffices to prove that $\Theta_{g-1}$ and $\Theta_g'$ are well defined. 
The first one is the easiest: as was previously observed, the functions $\{\theta_u(\tau;\,z)\}_{u\in\mathbb{Z}_2^{g-1}}$ defined in \eqref{eq:thetau} constitute a basis of $V(\tau)$ in \eqref{eq:secondorder}. 
Since the latter is base point free, at least one of the $\theta_u(\tau;\,0)$ must be nonzero, and hence
$\Theta_{g-1}$ is well defined.

We now show that $\Theta'_g$ is well defined.
Take any $j\,\in\,\mathbb{Z}_2$ and $u\,=\, (u_1,\,\cdots ,\, u_{g-1})\,\in\,\mathbb{Z}_2^{g-1}$. It is easy to see from the definition that
$$\theta\begin{bmatrix}
j & u\\
1 & 0
\end{bmatrix}(2\Pi;\, 2z)\ =\ \theta_{(j,u_1,\cdots ,u_{g-1})}(\Pi;\,z+e_1/4),$$
where $\theta_{(j,u_1,\cdots ,u_{g-1})}(\Pi,z)$ is defined in \eqref{eq:thetau}.
Since 
$$\{\theta_{(j,u_1,\cdots ,u_{g-1})}(\Pi;\cdot)\}_{j\in \mathbb{Z}_2,u\in\mathbb{Z}_2^{g-1}}\ = \ \{\theta_v(\Pi;\cdot)\}_{v\in\mathbb{Z}_2^g}$$ 
is a basis of the base point-free linear system $V(\Pi)$, we must have 
$$\theta\begin{bmatrix}
j & u\\
1 & 0
\end{bmatrix}(2\Pi;\, 0) \ =\ \theta_{(j,u_1,\cdots ,u_{g-1})}(\Pi;e_1/4)\,\neq \,0$$ for some $j\,\in\,\mathbb{Z}_2$ and some $u\,\in\,\mathbb{Z}_2^{g-1}$. However, by \eqref{eq:thetaparity}, $\theta\begin{bmatrix}
1 & u\\
1 & 0
\end{bmatrix}(2\Pi;\, 0)\,= \,0$ for all $u\in\mathbb{Z}_2^{g-1}$. 
This proves that $\theta\begin{bmatrix}
0 & u\\
1 & 0
\end{bmatrix}(2\Pi;\, 0)\,\neq\, 0$ for some $u\,\in\,\mathbb{Z}_2^{g-1}$, and it follows that $\Theta'_g$ is well defined.
\end{proof}

To simplify notation we will use the same symbol $\omega_{FS}$ to denote the Fubini--Study metric on any projective space with a fixed system of homogeneous coordinates.
\begin{corollary}
\label{cor:equivalentdebar}
 Identify the cotangent bundle of $\widetilde{M}_g$ with $f_\ast K^{\otimes 2}_{\mathcal{C}/B}$ through the isomorphism $\lambda$  described in \eqref{eq:lambdaiso}. Then
$$\debar \beta^P\ =8\pi\ \tau^\ast \Theta_{g-1}^\ast \omega_{FS}\ =\ 4\pi\ \tau^\ast (\theta^2_{g-1})^\ast \omega_{FS}\ =\ 4\pi\, \Pi^\ast (\Theta_{g
}^{SJ})^\ast \omega_{FS}\ =\ 8\pi \,\Pi^* \,(\Theta_g')^\ast\omega_{FS}.$$
\end{corollary}

\begin{proof}
The first equality is Theorem \ref{teo:debar}. 
Recall from  \eqref{eq:pyramid} that $p\circ v_2 \circ \Theta_{g-1} = \theta^2_{g-1}$.
Since $p$ is isometric it follows from   \eqref{eq:isometricveronese} that 
$(\theta^2_{g-1})^* \omega_{FS} = 2 \Theta_{g-1}^*\omega_{FS}$. This proves 
the second equality. 
The third follows immediately from Theorem \ref{TSJ}. The fourth equality follows again using \eqref{eq:pyramid}:  
$ (\Theta_{g
}^{SJ})^\ast \omega_{FS}\
= (\Theta'_{g-1})^* v_2^* p^* \omega_{FS} = 2 (\Theta'_{g-1})^* \omega_{FS}$.
\end{proof}

\subsection{The answer to Question \ref{question:one}}

The following notation will be used: for $Z\,\in\,\mathbb{H}_{k}$ with $k\,<\,g$, denote
$$\{Z\}\times \mathbb{H}_{g-k}\, :=\, \left\{\begin{bmatrix}
 Z & 0\\
 0 & Y
\end{bmatrix}\,\,\big\vert\,\, Y\,\in\, \mathbb{H}_{g-k}\right\}, \ \ \, \mathbb{H}_{g-k}\times \{Z\}\,:=\,\left\{\begin{bmatrix}
 Y & 0\\
 0 & Z
\end{bmatrix}\,\, \big\vert\,\, Y\,\in\, \mathbb{H}_{g-k}\right\}.$$

\begin{lemma}
\label{lemma:thetareducible}
Let $\varepsilon,\,\delta\,\in\, \mathbb{Z}_2^{k}$ and $\varepsilon',\,\delta'\,\in\, \mathbb{Z}_2^{g-k}$. Let $\tau\,\in\, \mathbb{H}_k$ and $\tau'\,\in \,\mathbb{H}_{g-k}$, and let $T\,=\,\begin{bmatrix}
 \tau & 0\\
 0 & \tau'
\end{bmatrix}$. Then 
\begin{multline}
\theta\begin{bmatrix}
 \varepsilon_1,\cdots ,\varepsilon_k, \varepsilon_1',\cdots ,\varepsilon_{g-k}'\\
 \delta_1,\cdots ,\delta_k, \delta_1',\cdots ,\delta_{g-k}'
\end{bmatrix}(T;\,z)\\
=\ \theta\begin{bmatrix}
 \varepsilon_1,\cdots ,\varepsilon_k\\
 \delta_1,\cdots ,\delta_k
\end{bmatrix}(\tau;\,(z_1,\cdots ,z_k))\theta\begin{bmatrix}
 \varepsilon_1',\cdots ,\varepsilon_{g-k}'\\
 \delta_1',\cdots ,\delta_{g-k}'
\end{bmatrix}(\tau';\, (z_{k+1},\cdots ,z_g)). 
\end{multline}
\end{lemma}

\begin{proof}
This follows immediately from the definition in \eqref{eq:thetachar}.
\end{proof}

\begin{proposition}
\label{prop:subsiegel}
Let $\Pi'\,\in\,\mathbb{H}_{g-1}$.
The restriction
$$\Theta'_g\big\vert_{\mathbb{H}_1\times \{\Pi'\}}\ :\ \mathbb{H}_1\times \{\Pi'\}\ \longrightarrow\ \mathbb{P}^{2^{g-1}-1}$$
is a constant map, while the restriction
$$\Theta'_g\big\vert_{\{\Pi'\}\times \mathbb{H}_1}\ :\ \{\Pi'\}\times \mathbb{H}_1
\ \longrightarrow\ \mathbb{P}^{2^{g-1}-1}$$
is an immersion. (See \eqref{b3} for the definition of $\Theta'_g$.)
\end{proposition}
\begin{proof}
Denote by $z$ the coordinate on $\mathbb{H}_1$. By Lemma \ref{lemma:thetareducible},
$$\Theta_g'\left(\begin{bmatrix}
 z & 0\\
 0 & \Pi'
\end{bmatrix}\right) \ = \ \left[\theta\begin{bmatrix}
 0\\
 1
\end{bmatrix}(2z;\,0)
\theta\begin{bmatrix}
 u\\
 0
\end{bmatrix}(2\Pi';\,0)\right]_{u\in\mathbb{Z}_2^{g-1}}.$$
Therefore all of $\mathbb{H}_1\times\{\Pi'\}$ is mapped to the point with homogeneous coordinates
$$\left[\theta\begin{bmatrix}
 u\\
 0
\end{bmatrix}(2\Pi';\,0)\right]_{u\in\mathbb{Z}_2^{g-1}}\ =\ [\theta_u(\Pi';0)]_{u\in \mathbb{Z}_2^{g-1}}\ =\ \Theta_{g-1}(\Pi').$$
(See  \eqref{Thg} for the definition of   $\Theta_{g-1}$).This proves the first statement.

For the proof of the second statement, the following notation is used: write $u\,
\in\, \mathbb{Z}_2^{g-1}$ as $(u',\,u_{g-1})$ with $u'\,\in\,\mathbb{Z}_2^{g-2}$ and $u_{g-1}\,\in\, \mathbb{Z}_2$. By Lemma \ref{lemma:thetareducible}, we have
\begin{equation}
\label{eq:secondrestriction}
\Theta_g'\left(\begin{bmatrix}
 \Pi' & 0\\
 0 & z
\end{bmatrix}\right) \ =\ \left[\theta\begin{bmatrix}
 0 & u'\\
 1 & 0
\end{bmatrix}(2\Pi';\,0)
\theta\begin{bmatrix}
 u_{g-1}\\
 0
\end{bmatrix}(2z;\, 0)\right]_{u\in\mathbb{Z}_2^{g-1}}.
\end{equation}
Consider then the following matrix with $2^{g-1}$ rows indexed by $\mathbb{Z}_2^{g-1}$ and $2$ columns indexed by $\mathbb{Z}_2$:
$$A_{u,v}\ :=\ \begin{cases}
 \theta\begin{bmatrix}
 0 & u'\\
 1 & 0
 \end{bmatrix}(2\Pi';\,0) & \text{if } v\,=\,u_{g-1}\\
 0 & \text{otherwise.}
\end{cases}$$
By \eqref{eq:secondrestriction}, the map $f\,:\,\mathbb{P}^1\, \longrightarrow\, \mathbb{P}^{2^{g-1}-1}$ induced by $A$ satisfies the following
$$\begin{tikzcd}
\mathbb{H}_1 \arrow[r, "\Theta_1"] \arrow[d, "\iota"'] & \mathbb{P}^2 \arrow[d, "f"] \\
\mathbb{H}_{g} \arrow[r, "\Theta'_g"'] & \mathbb{P}^{2^{g-1}-1} 
\end{tikzcd}$$
where $\iota$ it the map $z\, \longmapsto\, \begin{bmatrix}
 \Pi & 0\\
 0 & z
\end{bmatrix}$. It is clear from the definition that $A$ has the following form:
$$\begin{bmatrix}
 a_1 & 0\\
 0 & a_1\\
 \vdots & \vdots \\
 a_{2^{g-2}} & 0\\
 0& a_{2^{g-2}}
\end{bmatrix}$$
Consequently, $f$ is an embedding. The second part of the proposition then follows from the fact that $\Theta_1$ is an immersion, see \cite[p.~1490]{SM}.
\end{proof}

Consider the matrix 
\begin{equation}
\label{eq:thematrix}
M\ =\ \begin{bmatrix}
 A & 0\\
 0 & A
\end{bmatrix}\ \ \, \text{with}\ \ \, 
A\ =\ \begin{bmatrix}
 0 & I_{g-1}\\
 1 & 0
\end{bmatrix}
\end{equation}
($I_{g-1}$ is the $(g-1)\times (g-1)$ identity matrix). It is straightforward to check that $A^T\,=\,A^{-1}$, and that $M\,\in\, \operatorname{Sp}(2g,\mathbb{Z})$. The following is also immediate: if one denotes by $\cdot $ the action of $\operatorname{Sp}(2g,\mathbb{Z})$ on $\mathbb{H}_g$, then
\begin{equation}
\label{eq:mswaps}
M\cdot \begin{bmatrix}
 z & 0\\
 0 & \Pi'
\end{bmatrix}\ =\
\begin{bmatrix}
 \Pi' & 0\\
 0 & z
\end{bmatrix}\qquad \forall\,\ z\,\in\,\mathbb{H}_1,\,\forall\,\ \Pi'\,\in\,\mathbb{H}_{g-1}.
\end{equation}
Let $M\,:\,\mathbb{H}_g\, \longrightarrow\, \mathbb{H}_g$ be the map induced, by $M$, on
$\mathbb{H}_g$ through the action of $\operatorname{Sp}(2g,\mathbb{Z})$ on $\mathbb{H}_g$. By \eqref{eq:mswaps}, $M$ maps $\mathbb{H}_1\times \{\Pi'\}$ isomorphically onto $\{\Pi'\}\times \mathbb{H}_1$.

\begin{corollary}
\label{cor:nottoag}
The form $(\Theta'_g)^\ast \omega_{FS}$ on $\mathbb{H}_g$ (see \eqref{b3})
does not descend to $A_g$.
\end{corollary}

\begin{proof}
Fix $\Pi'\,\in\, \mathbb{H}_{g-1}$. To argue by contradiction, 
assume that $(\Theta')^\ast \omega_{FS}$ descends to $A_g\,=\,\mathbb{H}_g/\operatorname{Sp}(2g,\mathbb{Z})$. 
Then $(\Theta'_g)^\ast \omega_{FS}$ must be invariant under the pullback using
the action of $M$ in \eqref{eq:thematrix} on $\mathbb{H}_g$. Since $M$ maps $\mathbb{H}_1\times \{\Pi'\}$ isomorphically onto $\{\Pi'\}\times \mathbb{H}_1$, 
$$(\Theta'_g)^\ast \omega_{FS}\big\vert_{\{\Pi'\}\times \mathbb{H}_1}\ =\ M^\ast(\Theta'_g)^\ast \omega_{FS}\big\vert_{\{\Pi'\}\times \mathbb{H}_1}\ =\
(\Theta'_g)^\ast \omega_{FS}\big\vert_{ \mathbb{H}_1\times \{\Pi'\}}.$$
But this is impossible because the left-hand side is nonzero by the second part of Proposition \ref{prop:subsiegel}, while the right-hand side is zero by the first part of the same proposition.
In view of this contradiction, we conclude that $(\Theta'_g)^\ast \omega_{FS}$
does not descend to $A_g$.
\end{proof}

\begin{lemma}
\label{lemma:limitsequence}
Let $y_0\,\in\,\mathbb{H}_1$ correspond to an elliptic curve $E_{y_0}$ with $\operatorname{Aut}(E_{y_0})\, =\, \{\pm 1\}$. Let 
$C'$ be a Riemann surface of genus $g-1$ whose Jacobian $J(C')$ is simple
with $\operatorname{Aut}(J(C'))\,=\,\{\pm 1\}$. 
Denote by 
$\Pi'\,\in\,\mathbb{H}_{g-1}$ a period matrix of $J(C')$, and
set $$\widetilde{y}_0\ =\ \begin{bmatrix}
 y_0 & 0\\
 0 & \Pi'
\end{bmatrix}.$$
Then there are points $x_n=[C_n]\,\in\, \widetilde M_g$ and vectors $w_n\, \in\, T_{x_n}\widetilde M_g$ such that
$$\lim_{n\to \infty} \Pi(x_n) = \widetilde{y}_0, \qquad
\lim_{n\to \infty} d\Pi (w_n)\ =\ w,$$
where $w$ is a generator of the tangent space $T_{\widetilde{y}_0} (\mathbb{H}_1\times \{\Pi'\})$.
\end{lemma}

\begin{proof}
Fix a point $p\in C'$. 
For $y\in\mathbb{H}_1$ set
$$E_y \ :=\ \frac{\mathbb{C}}{\mathbb{Z}+y\mathbb{Z}},$$
and denote by $E_y\cup_p C'$ the curve of compact type obtained by identifying the point $p$ of $C'$ with the origin of $E_y$.
Fix $p\,\in\, C'$. Fix a Kuranishi family $\phi\,:\, \mathcal X \, \longrightarrow\, B$,
where $B$ is a polydisk centered at the origin, and the central fibre is isomorphic to $E_{y_0}\cup_pC'$. Since $E_{y_0}\cup_p C'$ is a curve of compact type, we can assume that $B \subset \overline{M}_{g} -  \Delta_0$.
Let $$\bar j\ :\ B\ \longrightarrow\ A_g\qquad b\longmapsto J(\phi^{-1}(b))$$ denote the extended period map. 
Locally there is a lift 
$$\widetilde{j} \,:\, B \,\longrightarrow \,\mathbb H_g$$ 
of $\bar j$
such that $\widetilde{j} (0) \,=\, \widetilde{y}_0$. 
It can be  constructed as follows: first one fixes the symplectic basis of $H_1(C',\mathbb Z)$ with respect to which  $\Pi'$ is the period matrix of $C'$. Next 
one adds to this a symplectic basis of $H_1(E_{y_0}, \mathbb Z)$.
The result is a basis
$$\mathcal{B}_0 \ =\ \{a_1(0),\,\cdots ,\, a_g(0),\, b_1(0),\, \cdots ,\, b_g(0) \}$$ 
of $H_1(\phi^{-1}(0),\mathbb{Z})$, which extends 
to a symplectic basis
$$\mathcal{B}(b)\ =\ \{a_1(b),\, \cdots ,\, a_g(b),\, b_1(b),\, \cdots ,\, b_g(b)\}$$
of $H_1(\phi^{-1}(b),\, \mathbb{Z})$ for all $b\in B$.
Set
\begin{equation}
\label{eq:defB0}
B_0 \,:= \, \{b\in B\,\big\vert \,\phi^{-1}(b)\text{ is smooth}\} \,\subset \, B.
\end{equation} 
By sending $b\in B_0$ to the pair $(\phi^{-1}(b), \mathcal{B}(b))$ we obtain a map $\eta\,:\, B_0 \,\longrightarrow \,\widetilde{M}_g$ to the Torelli space.
By construction, we have
\begin{gather}
\label{tildejB}
 \widetilde{j} \big\vert_{B_0}\ =\ \Pi \circ \eta,
\end{gather}
where $\Pi\,:\,\widetilde{M}_g\, \longrightarrow\, \mathbb{H}_g$ is the period map \eqref{defPi}.

Next, consider the family $\mathcal E \,\longrightarrow\, \mathbb H_1$, with fiber $E_y\cup _p C'$ over $y$.
There is a map $\psi \,:\, U \,\longrightarrow \,B$ defined on an open neighbourhood $U$ of $y_0$ in $\mathbb H_1$ such that $\mathcal E \,\cong\, \psi^*\mathcal X$.
In particular, we have $\psi(y_0)\,=\,0\,\in\, B$.

We claim that
\begin{equation}\label{d1}
\widetilde{j} \circ \psi\ =\ \alpha,
\end{equation}
where $\alpha$ is the map
\begin{gather*}
\alpha\ :\ U\ \longrightarrow\ \mathbb{H}_g,\qquad \alpha (y)\ :=\ \begin{bmatrix}
 y & 0 \\ 0 & \Pi'
 \end{bmatrix}.
\end{gather*}
To prove this, note that for each $y\in U$ we have
$$\bar j\circ\psi(y)\, =\, \bar j ([E_y \cup_p C'])\, =\, [E_y \times JC']\, =\, [\alpha(y)],$$
where $[x]$ denotes the class of $x\in\mathbb{H}_g$ in $A_g$. 
Since $\widetilde{j}$ lifts $\bar j$,  we also have  
\begin{gather*}
    \left[\widetilde{j} \circ \psi(y)\right] = \bar {j} \circ \psi (y) =[\alpha(y)].
\end{gather*}
 Hence 
 there is a $\gamma_y \,\in \,\operatorname{Sp}(2g,\mathbb{Z})$ such that $\gamma_y \cdot\alpha(y)\, =\, \widetilde{j} \circ \psi (y)$. 
Since $\operatorname{Sp}(2g,\mathbb Z)$ is countable, Baire's theorem ensures that $\gamma_y=\gamma $ is independent of $y$. 
Moreover, $\gamma \cdot \widetilde{y}_0\, =\, \widetilde{y}_0$, because $\alpha(y_0)\, =\, \widetilde{y}_0\,=\,\widetilde{j} (0)\, =\,
\widetilde{j} \psi(y_0)$. 
By the assumptions on $y_0$ and $J(C')$ the abelian variety $E_{y_0}\times J(C')$ does not have automorphisms.
In particular, the stabilizer of $\alpha(y_0)$ in $\operatorname{Sp}(2g,\mathbb Z)$ is trivial. 
Hence $\gamma \,=\,1$, which proves \eqref{d1}

Consider a generator $\nu$ of $T_{y_0}{\mathbb H}_1$ , and set
$v\,:=\, d\psi(\nu) \,\in\, T_0B$. 
As $\widetilde{j} \circ \psi\, =\, \alpha$, we have $d\widetilde{j} (v)\, =\, d\alpha(\nu)$.
Then $d\widetilde{j} (v) \,\neq\, 0$ since $\alpha$ is an embedding. 
Pick a sequence of points $x_n \,\in\, B_0$ 
and one of tangent vectors $v_n \,\in\, T_{x_n}B_0$ such that $x_n \,\to\, 0$ and $v_n \,\to\, v$.
Put $w_n\,:=\, d\eta(v_n)$. Then by \eqref{tildejB}, 
$$d\Pi(w_n ) \,=\,d\Pi (d\eta (v_n))\,=\, d\widetilde{j} (v_n)\ \to\ d\widetilde{j} (v)\, =\,d
\alpha(\nu).$$
Since $w \,:=\, d\alpha(\nu)$ is a generator of $T_{\widetilde{y}_0}({\mathbb H}_1 \times \{\Pi'\})$, the lemma is proved.
\end{proof}

\begin{proposition}
The section $\debar\beta^P$ constructed in \eqref{eq:defbetap}
does not descend to $M_g$.
\end{proposition}

\begin{proof}
Assume by contradiction that $\debar \beta^P$  descends to the quotient $M_g\,=\,\widetilde{M}_g/\operatorname{Sp}(2g,\mathbb{Z})$. Then $\frac{1}{8\pi}\debar\beta^P\,=\,\Pi^\ast(\Theta')^\ast \omega_{FS}$ must be $M$-invariant, where $M:\widetilde{M}_g\,\longrightarrow\, \widetilde{M}_g$ the map is induced by the matrix in \eqref{eq:thematrix}. Since
 the period map $\Pi:\widetilde M_g\longrightarrow \mathbb{H}_g$ is $\operatorname{Sp}(2g, \mathbb{Z})$--equivariant, we get
$$0\ =\ \Pi^\ast(\Theta')^\ast \omega_{FS} - M^\ast\Pi^\ast(\Theta')^\ast \omega_{FS}\ = \ \Pi^\ast\left((\Theta')^\ast \omega_{FS} - M^\ast(\Theta')^\ast \omega_{FS}\right).$$
Fix a Riemann surface $C'$ as in the hypothesis of Lemma \ref{lemma:limitsequence}, and let $\Pi'$ be a period matrix of $C'$.
Choose a generic point $y_0\,\in\,\mathbb{H}_1$. The point 
$y_0$ satisfies the hypothesis of Lemma \ref{lemma:limitsequence}, and therefore 
there is a sequence of tangent vectors $w_n\,\in\, T \widetilde M_g$ such that 
$\lim d\Pi (w_n)\,=\,w$, where $w$ is a generator of the tangent space $T (\mathbb{H}_1\times\{\Pi'\})$ at the point $\begin{bmatrix}
 y_0 & 0\\
 0 & \Pi'
\end{bmatrix}$.
Then we obtain the following:
$$0\,=\,\left((\Theta')^\ast \omega_{FS} - M^\ast(\Theta')^\ast \omega_{FS}\right)\left(d \Pi w_n, \,\overline{d \Pi w_n}\right).$$
Taking the limit $n\,\to\, \infty$ we get that
$$M^\ast (\Theta')^\ast \omega_{FS}(w,\, \overline{w})\ =\ (\Theta')^\ast \omega_{FS}(w,\,\overline{w}).$$
The last equation is true for the generic $w\,\in\, T(\mathbb{H}_1\times \{\Pi'\})$. Therefore,
$$(\Theta')^\ast \omega_{FS}\big\vert_{\{\Pi'\}\times \mathbb{H}_1}\ =\ M^\ast(\Theta')^\ast \omega_{FS}\big\vert_{\{\Pi'\}\times \mathbb{H}_1}\ =\ (\Theta')^\ast \omega_{FS}\big\vert_{ \mathbb{H}_1\times \{\Pi'\}}.$$
This is impossible by the same argument in  the proof of Corollary \ref{cor:nottoag}.
In view of this contradiction, the proof of the proposition is complete.
\end{proof}

\begin{corollary}
The answer to Question \ref{question:one} is negative.
\end{corollary}
\begin{proof}
    Since $\debar \beta^P$ does not descend to $M_g$, neither does $\beta^P$. In other words $\beta^P$ is not a pullback of a section of $V_g$ via the forgetful map $\xi: R_g \lra M_g$.
    In particular there exist  curves $C$ with two distinct étale covers $\pi_i\,:\,\widetilde C_i\,
 \longrightarrow\, C$, for $i\,=\,1,\,2$, such that
$$\beta^P_{\pi_1} \ \neq\ \beta^P_{\pi_2}.$$
\end{proof}


\begin{thebibliography}{ZZZZZ}

\bibitem[ACGH]{ACGH2} E. Arbarello, M. Cornalba and P. B. Griffiths, \textit{Geometry of algebraic curves: volume II with a contribution by Joseph Daniel Harris}, Vol. 268. Springer Science and Business Media, 2011.

\bibitem[Be]{Be}
A. Beauville, Prym varieties: a survey. Proc. of Symp. in Pure Mat. Vol. 49. No. 1. 1989.

\bibitem[Be2]{Be2}
A. Beauville, {\it Variétés de Prym et jacobiennes intermédiaires}, Annales scientifiques de l'École Normale Supérieure. Vol. 10. No. 3. 1977.

\bibitem[BD]{BD}
A. Beauville and O. Debarre. {\it Sur le probleme de Schottky pour les variétés de Prym}, Annali della Scuola Normale Superiore di Pisa-Classe di Scienze 14.4 (1987): 613-623.


\bibitem [BGT]{BGT} I.~Biswas, A.~Ghigi and C.~Tamborini, \newblock {Theta
bundle, Quillen connection and the Hodge theoretic projective
structure}, \textit{Commun. Math. Phys.}  (2024) 405:243.

\bibitem[BL]{BL} C.~Birkenhake and H.~Lange, \newblock
 {\em Complex abelian varieties}, volume 302 of {\em Grundlehren der
 Mathematischen Wissenschaften}. \newblock Springer-Verlag,
 Berlin, second edition, 2004.

\bibitem[BCFP]{BCFP} I.~ Biswas, E.~Colombo, P.~Frediani and
 G. P. Pirola, Hodge theoretic projective structure on Riemann
 surfaces, {\it Jour. Math. Pures. Appl.} {\bf 149} (2021), 1--27.


\bibitem[BFPT]{BFPT} I.~ Biswas, F. F.~Favale, G. P.~Pirola and S.~Torelli,
  \newblock {Quillen connection and the uniformization of Riamann
surfaces}, \newblock {\it Annali di Matematica Pura ed Applicata}
  {\bf 201} (2022), 2825--2835


\bibitem [BGV]{BGV} I.~Biswas, A.~Ghigi and L.~Vai, \newblock {Theta
functions and projective structures}, {\it Int. Math. Res. Not.} Vol. 2025,
Issue 1, rnae271.

\bibitem[BV]{BV} I.~Biswas and L.~Vai, On the projective structures given by theta, Arxiv [Preprint] (2025). Arxiv, https://arxiv.org/abs/2504.15866

\bibitem[BR]{BR1} I. Biswas and A. K. Raina, Projective structures on
 a Riemann surface, {\it Inter. Math. Res. Not.} (1996), No. 15,
 753--768.

\bibitem[Fa]{FAY}
J. D. Fay, \textit{Theta functions on Riemann surfaces}, Vol. 352. Springer, 2006.

\bibitem[FR]{FR} H. M. Farkas and H. E. Rauch, Period relations of Schottky type on Riemann surfaces, {\it Ann. of Math.} {\bf 92} (1970), 434--461.

\bibitem[Ig]{IG} J.-I. Igusa, {\it Theta functions}, Vol. 194. Springer Science and Business Media, 2012.

\bibitem[GH]{GH} P. B. Griffiths and J. Harris, {\it Principles of algebraic geometry}, John Wiley and Sons, 2014.

\bibitem[GSM]{GSM} S. Grushevsky and R. Salvati Manni,
Gradients of odd theta functions, {\it Journal reine ang. Math.} {\bf 573} (2004), 45--59.


\bibitem[Gu1]{GU} R. C. Gunning, Special coordinate coverings of Riemann surfaces, {\it Math. Ann.} {\bf 170} (1967), 67--86.

\bibitem[Gu2]{GuLec} R. C. Gunning, {\it Lectures on Riemann surfaces}, Princeton University Press, 1976.

\bibitem[Hu]{Hu} J. H. Hubbard, The monodromy of projective structures,
{\it Riemann surfaces and related
topics: Proceedings of the 1978 Stony Brook Conference} (State Univ. New York, Stony
Brook, N.Y., 1978), 257--275, Ann. of Math. Stud., Vol. 97, Princeton Univ. Press, Princeton, N.J.
1981.

\bibitem[Iz]{IZ}
E. Izadi, Second order theta divisors on Pryms, {\it Bull. Soc. Math. Fr.} {\bf 127} 
(1999), 1--23.

\bibitem[IP]{IP}
E. Izadi and C. Pauly, Some Properties of Second Order Theta Functions o Prym Varieties, {\it Math. Nach.} {\bf 230} (2001), 73--91.

\bibitem[Mu1]{Mu}
D. Mumford, Prym varieties I, {\it Contributions to analysis}, pp. 325--350, Academic Press, 1974.

\bibitem[Mu2]{Mu2} D. Mumford, \textit{Abelian varieties}, Tata Institute of Fundamental Research Studies in Mathematics, vol. 5, American Mathematical Society, 2008.

\bibitem[Na]{Na}
J. C. Naranjo, Fourier transform and Prym varieties, {\it Jour. Reine Ang. Math.} (2003),
221-230.

\bibitem[Po]{Po} C. Poor, Fay's trisecant formula and cross-ratios,
  {\em Proc. Amer Math Soc} \textbf{114} (1992), 667--671.

\bibitem[SM]{SM}
R. Salvati Manni, Modular varieties with level 2 theta structure, {\it Amer. Jour. Math.} 
{\bf 116} (1994), 1489--1511.

\bibitem [Ty]{Tyu} A. N. Tyurin, On periods of quadratic
 differentials, {\em Russian Math. Surveys} {\bf 33} (1987),
 169--221.

\bibitem[vG]{VG} B. van Geemen, The Schottky problem and second order theta functions, 
{\it Workshop on Abelian Varieties and Theta Functions} (Spanish) (Morelia, 1996), 1998.


		
\bibitem[ZT]{ZT} P. G. Zograf and L. A. Takhtadzhyan, On the
uniformization of Riemann surfaces and on the Weil-Petersson metric
on the Teichm\"uller and Schottky spaces, {\it Math. USSR-Sb.} {\bf
60} (1988), 297--313.


\end{thebibliography}
\end{document}